\documentclass[11pt]{article}

\usepackage[top=1in,right=1in,left=1in,bottom=1in]{geometry}
\usepackage{amsthm,amsmath,amsfonts,amssymb,colortbl,float,graphicx}
\usepackage[numbers,sort&compress]{natbib}
\usepackage{enumerate,array,multirow,bbm}
\usepackage{booktabs,siunitx}
\usepackage{xcolor}

\newtheorem{Theorem}{Theorem}[section]

\newtheorem{Proposition}{Proposition}[section]
\newtheorem{Remark}{Remark}[section]
\newtheorem{Lemma}{Lemma}[section]

\numberwithin{equation}{section}

\newcommand{\bx}{\textbf{x}}

\begin{document}

\title{A Uniformly Accurate Multiscale Time Integrator for the Klein-Gordon-Schr\"odinger Equations in the Nonrelativistic Regime via Simplified Transmission Conditions}

\author{Yue Feng \footnote{School of Mathematics and Statistics,
  Xi'an Jiaotong University,
  Xi'an 710049, P. R. China. yue.feng@xjtu.edu.cn.}
	~~~and~~~
	Caoyi Liu \footnote{Corresponding author. Department of Mathematics, National University of Singapore, Singapore 119076, Singapore. e1353374@u.nus.edu.}
}

\maketitle

\begin{abstract}
We propose a novel and simplified multiscale time integrator Fourier pseudospectral (MTI-FP) method for the Klein-Gordon-Schr\"odinger (KGS) equations with a dimensionless parameter $ \varepsilon \in (0,1]$, where $\varepsilon$ is inversely proportional to the speed of light. The proposed MTI-FP method is rigorously proved to achieve uniform first-order accuracy in time in the nonrelativistic regime, i.e., as $\varepsilon \to 0^+$. In this regime, the solution of the KGS equations exhibits temporal oscillations with an $O(\varepsilon^2)$-wavelength, imposing stringent resolution requirements on classical numerical methods. The uniformly accurate MTI-FP method is built upon two key points: (i) a multiscale decomposition by frequency in each time interval with simplified transmission conditions, and (ii) an exponential integrator for temporal discretization combined with the Fourier pseudospectral method for spatial discretization. Using the energy method and mathematical induction, we rigorously establish two independent error bounds in $H^1$-norm at $O(h^{m_0 -1} + \tau^2 / \varepsilon^2 )$ and $O(h^{m_0 -1} +  \varepsilon^2)$ with mesh size $h$, time step $\tau$ and $m_0$ an integer dependent on the regularity of the solution. These estimates imply that the MTI-FP method converges uniformly and optimally in space, and uniformly in time at $O(\tau)$ with respect to $\varepsilon \in (0, 1]$. Furthermore, by incorporating a linear interpolation of the micro-variables with the multiscale decomposition in each time interval, we obtain a uniformly accurate numerical solution for any time $t>0$. Consequently, the proposed MTI-FP method has a super-resolution property in time from the perspective of Shannon sampling theory. Ample numerical experiments are provided to validate the error estimates and to demonstrate the super-resolution property. Finally, the method is applied to numerically investigate the convergence rates of the KGS equations to different limiting models.



\end{abstract}

\section{Introduction}\label{sec:intro}
In this paper, we consider the following coupled Klein-Gordon-Schr\"odinger (KGS) equations in $d$-dimensions ($d = 3, 2, 1$) which are used as a mean-field model for nucleon-meson interaction through the Yukawa coupling \cite{fukuda1978coupled,yukawa1935interaction}
\begin{equation}
\label{KGSdim}
\left\{
\begin{aligned}
& i \hbar \partial_t \psi (\bx,t) + \frac{\hbar^2}{2 m_1} \Delta \psi(\bx,t) + \eta \phi(\bx,t) \psi(\bx,t) = 0, \quad \bx \in \mathbb{R}^d, \quad d = 1,2,3, \\
& \frac{1}{c^2} \partial_{tt} \phi (\bx,t) - \Delta \phi(\bx,t) + \frac{m_2^2 c^2}{\hbar^2 } \phi(\bx,t) - \eta|\psi(\bx,t)|^2 = 0, \quad t >0,
\end{aligned}\right.
\end{equation}
where $\psi(\bx,t)$ represents a complex scalar nucleon field and $\phi(\bx,t)$ is a real scalar meson field. Here, $\bx \in \mathbb{R}^d$ is the spatial coordinate, $t$ is time,  $\hbar$ is the Planck constant, $0 \neq\eta \in \mathbb{R}$ is the Yukawa coupling constant, $c$ is the speed of light and $m_1$, $m_2>0$ are the mass of the nucleon and the meson, respectively. The KGS system is widely used in many physical fields, including the quantum field theory \cite{fukuda1978coupled,yukawa1935interaction}, plasmas physics \cite{hasegawa1995solitons,shukla2010nonlinear} and nonlinear optics \cite{kivshar2003optical}.  

In order to non-dimensionalize the KGS system \eqref{KGSdim}, we introduce
\begin{equation}
\tilde \bx = \frac{\bx}{x_s},\qquad
\tilde t = \frac{t}{t_s}, \qquad \tilde \psi(\tilde \bx,\tilde t) = x_s^{d/2} \psi(\bx,t), \qquad \tilde \phi(\tilde \bx,\tilde t) =  \displaystyle \frac{\phi(\bx,t)}{\phi_s}, \label{dim}
\end{equation}
where $x_s$, $t_s = \frac{2 m_1 x_s^2}{\hbar}$ and $\phi_s = \frac{\hbar x_s^{-d/2}}{\sqrt{2 m_1}}$ are the length unit, time unit and meson field unit, respectively. Plugging \eqref{dim} into \eqref{KGSdim} and removing all `$\sim$', we obtain the following dimensionless KGS equations in $d$-dimensions ($d = 3, 2, 1$):
\begin{equation}
\label{KGSndim}
\left\{\begin{aligned}
 & i \partial_t \psi(\bx,t) + \Delta \psi(\bx,t) + \lambda \phi(\bx,t) \psi(\bx,t) = 0, \quad \bx \in \mathbb{R}^d, \quad t > 0, \\
        & \varepsilon^2 \partial_{tt} \phi(\bx,t) - \Delta \phi(\bx,t) +  \frac{\mu^2}{\varepsilon^2} \phi(\bx,t) - \lambda | \psi(\bx,t)|^2 =0, 
\end{aligned}\right.
\end{equation}
where $\mu := \frac{m_2}{2 m_1}>0$ is the mass ratio, $\lambda = \frac{\eta \sqrt{2 m_1} x_s^{2-d/2}}{\hbar}$ and $\varepsilon := \frac{v}{c} = \frac{\hbar}{2cm_1x_s}$ is the ratio of the wave speed $v = \frac{x_s}{t_s} = \frac{\hbar}{2 m_1 x_s}$ and the speed of light. 
We remark here that the choice of the length unit $x_s$ is crucial to the nondimensionalization of the KGS system \eqref{KGSndim}, decides the observation scale of the time evolution of the particles, and determines (i) which phenomenon can be resolved numerically on prescribed spatial-temporal grids, and (ii) which dynamical features can be captured by asymptotic analysis. In particular, taking $x_s= \frac{\hbar}{2c m_1}$ ($\Longleftrightarrow  \varepsilon = 1$) corresponds to the classical regime, where the wave speed is of the same order as the speed of light $c$. On the other hand, when the wave speed is much smaller than the speed of light, it is more appropriate to select a different length scale $x_s$ such that $0<\varepsilon\ll1$, leading to the nonrelativistic limit regime.  

To study the dynamics of the KGS equations \eqref{KGSndim}, the initial data is usually given as
\begin{equation}
    \psi(\bx,0) = \psi_0(\bx),\quad \phi(\bx,0) = \phi_0(\bx), \quad \partial_t \phi(\bx,0) = \displaystyle \frac{1}{\varepsilon^2} \phi_1(\bx), \quad \bx \in \mathbb{R}^d, 
\label{KGSini}
\end{equation}
where the complex-valued function $\psi_0$ and the real-valued functions $\phi_0$ and $\phi_1$ are independent of $\varepsilon$. The KGS system \eqref{KGSndim} is dispersive, time symmetric and conserves the {\sl mass} of the nucleon field
\begin{align}
\mathcal{M}(t) = \left\| \psi(\cdot,t) \right\|_{L^2}^2 := \int_{\mathbb{R}^d} |\psi(\bx,t)|^2 d \bx \equiv  \int_{\mathbb{R}^d} |\psi_0(\bx)|^2 d \bx, \quad t \geq 0,
\end{align}
and the {\sl Hamiltonian} or total {\sl energy}
\begin{align}
\label{energy}
\mathcal{E}(t)
&:= \int_{\mathbb{R}^d} 
\left[ \frac12 \left( \varepsilon^2 |\partial_t \phi|^2
      + |\nabla \phi|^2 
      + \frac{\mu^2}{\varepsilon^2} |\phi|^2 \right)
      + |\nabla \psi|^2 - \lambda |\psi|^2 \phi \right] \, \mathrm{d} \bx \nonumber \\
&\equiv \int_{\mathbb{R}^d}
\left[ \frac12 \left( \frac{1}{\varepsilon^2} |\phi_1|^2
      + |\nabla \phi_0|^2
      + \frac{\mu^2}{\varepsilon^2} |\phi_0|^2 \right)
      + |\nabla \psi_0|^2 - \lambda |\psi_0|^2 \phi_0 \right] \, \mathrm{d} \bx \nonumber \\
& = \mathcal{E} (0), \qquad t \ge 0 . 
\end{align}

In the classical regime, i.e., $\varepsilon = O(1)$, the KGS equations \eqref{KGSndim} have been well studied both analytically and numerically. For well-posedness and regularity, we refer to literatures \cite{biler1990attractors,fukuda1978coupled,fukuda1975yukawa,guo1982global,boling1997attractor,guo1995global,hayashi1987global} and the references therein. For numerical approximations, many efficient and accurate numerical methods have been proposed and analyzed, including the finite difference method \cite{pan2013high,wang2014optimal,wang2018unconditional,zhang2010finite,zhang2005convergence}, the time-splitting method \cite{bao2007efficient}, the symplectic and multi-symplectic methods \cite{hong2009explicit,kong2013multisymplectic,kong2010semi}, the Chebyshev pseudospectral multidomain method \cite{dehghan2012numerical} and the Galerkin finite element methods \cite{yang2021unconditional}.

However, in the nonrelativistic limit regime, i.e., $0 < \varepsilon \ll 1$, the analysis and numerical computation of the KGS system \eqref{KGSndim} become significantly more complicated. The main difficulty arises from the fact that the solution exhibits highly oscillatory behavior in time, and the energy $\mathcal{E}(t) = O(\varepsilon^{-2})$ in \eqref{energy} becomes unbounded as $\varepsilon \to 0$. Taking into account the attractors and asymptotic behavior of this system in the nonrelativistic regime \cite{biler1990attractors,guo1995global,li2003asymptotic,lu2001global,ozawa1994asymptotic}, the solution of the Klein-Gordon equation can be decomposed as \cite{bao2012uniformly,machihara2002nonrelativistic,masmoudi2002nonlinear}
\begin{equation}
    \phi(\bx,t) = e^{i\mu t / \varepsilon^2} v(\bx,t) + e^{-i\mu t/ \varepsilon^2} \overline{v(\bx,t)} + o(\varepsilon), \quad \bx \in \mathbb{R}^d, \quad t \ge 0, \label{ansatzRd}
\end{equation}
where $\bar{v}$ denotes the complex conjugate of a complex-valued function $v$. By plugging \eqref{ansatzRd} into \eqref{KGSndim}, $(\psi,v)$ satisfies the Sch\"{o}dinger equations as a limiting model
\begin{equation}
\label{lim}
\left\{\begin{aligned}
& i \partial_t \psi(\bx,t) + \Delta \psi(\bx,t) + \lambda \left( e^{i\mu t / \varepsilon^2} v(\bx,t) + e^{-i\mu t/ \varepsilon^2} \overline{v(\bx,t)} \right) \psi(\bx,t) = 0,  \\
& 2i \partial_t v(\bx,t)  - \Delta v(\bx,t) =0, \quad \bx \in \mathbb{R}^d, \quad t > 0,
\end{aligned}\right.
\end{equation}
with initial data
\begin{align}
    \psi(\bx,0) = \psi_0(\bx), \qquad v(\bx,0) = \frac{1}{2}\Bigl[ \phi_0 (\bx) - \frac{i}{\mu} \phi_1(\bx) \Bigr], \label{lim_initial}
\end{align} 
or the Sch\"{o}dinger equations with wave operator as a semi-limiting model
\begin{equation}
\label{slim}
\left\{\begin{aligned}
& i \partial_t \psi(\bx,t) + \Delta \psi(\bx,t) + \lambda \left( e^{i\mu t / \varepsilon^2} v(\bx,t) + e^{-i\mu t/ \varepsilon^2} \overline{v(\bx,t)} \right) \psi(\bx,t) = 0,  \\
& 2i \partial_t v(\bx,t) + \varepsilon^2 \partial_{tt} v(\bx,t) - \Delta v(\bx,t) =0,\quad \bx \in \mathbb{R}^d, \quad t > 0, 
    \end{aligned}\right.
\end{equation}
with initial data
\begin{align}
\psi(\bx,0) = \psi_0(\bx), \quad v(\bx,0) = \frac{1}{2}\Bigl[ \phi_0 (\bx) - \frac{i}{\mu} \phi_1(\bx) \Bigr], \quad \partial_t v(\bx,0) = \gamma(\bx), \label{slim_initial}
\end{align}
where $\gamma(\bx)$ is chosen as the well-prepared initial data in the literatures \cite{bao2017uniformly}.

Based on these results, it is clear that, in the nonrelativistic regime, the solution $\phi(\bx,t)$ of the second equation in \eqref{KGSndim} exhibits highly oscillatory  behavior in time with $O(\varepsilon^2)$-wavelength. This poses significant challenges for the design and analysis of efficient and accurate numerical methods for the KGS system \eqref{KGSndim}-\eqref{KGSini}. In this regime, classical numerical methods require a prohibitively fine meshing strategy (or $\varepsilon$-scalability) to correctly resolve the oscillatory solution when $0 <\varepsilon \ll 1$. For example, finite difference time domain (FDTD) methods perform well when $\varepsilon = O(1)$, but in the nonrelativistic regime they necessitate a meshing strategy with $h = O(1)$ and $\tau = O(\varepsilon^3)$ \cite{su2018error}. Consequently, these methods are {\bf under-resolution} in time according to Shannon sampling theory, as they require $O(\varepsilon^{-1})\gg1$ grid points per wave when $0<\varepsilon\ll1$. In recent years, several uniformly accurate numerical methods have been proposed and analyzed, including multiscale time integrators (MTI) \cite{bao2017uniformly}, two-scale formulation (TSF) methods \cite{chartier2015uniformly}, asymptotic consistent exponential-type integrators \cite{baumstark2018asymptotic}, nested Picard iterative integrators (NPI) \cite{cai2023uniformly} and uniformly accurate integrators (UAI) \cite{calvo2023uniformly}. All of these methods share the same meshing strategy with $h=O(1)$ and $\tau=O(1)$ in the nonrelativistic regime, and thus they are {\bf super-resolution} methods in time according to Shannon sampling theory. Therefore, they yield accurate solutions even when the time step is much larger than the temporal wavelength of $O(\varepsilon^2)$.

The main aim of this paper is to propose a novel and simplified multiscale time integrator Fourier pseudospectral (MTI-FP) method for the KGS system \eqref{KGSndim}-\eqref{KGSini} by adapting a multiscale decomposition by frequency in each time interval with simplified transmission conditions. We carry out two error bounds in $H^1$-norm at $O(h^{m_0 -1} + \tau^2 / \varepsilon^2 )$ and $O(h^{m_0 -1}  + \varepsilon^2)$ with $m_0$ an integer dependent on the regularity of the solutions by adapting two distinct analytical techniques, i.e., the energy method and the mathematical induction, respectively. From these two independent error bounds, we immediately obtain a uniformly accurate error bound at $O(h^{m_0 -1} + \tau)$ with respect to $\varepsilon\in(0,1]$, which indicates that the proposed MTI-FP method possesses super-resolution property in time. Compared with the multiscale time integrators (MTIs) in the literature \cite{bao2017uniformly}, the proposed MTI-FP offers three main advantages: (i) the numerical scheme is significantly simplified, which greatly reduces the computational cost in practice; (ii)
the regularity requirement on the solution is weakened while remaining the same uniform first-order error bound in time, and (iii) it achieves optimal spatial accuracy given the regularity of the solution. Furthermore, by performing a linear interpolation of the micro variables in the multiscale decomposition \eqref{ansatzRd} within each time interval, we obtain a uniformly accurate approximation of the solution for any $t\ge0$ with almost no additional computational cost. This post-processing technique for constructing numerical solutions at arbitrary time $t\ge0$ with almost no additional computational cost is only applied to the MTI-FP methods and does not apply to other uniformly accurate numerical methods, such as the TSF methods and the UAI found in the literature.

The remainder of this paper is organized as follows. In Section~\ref{sec:multi}, we introduce the multiscale decomposition for the KGS system \eqref{KGSndim}-\eqref{KGSini} with simplified transmission conditions. The corresponding MTI-FP method is proposed in Section~\ref{sec:MTI} and its rigorous error estimates are established in Section~\ref{sec:unierr}. In Section~\ref{sec:interpolation}, we present and analyze a multiscale interpolation for global time $t > 0$. Numerical results are reported in Section~\ref{sec:numerical} to verify our error bounds and to numerically demonstrate the convergence rates of the KGS to its different limiting models. Finally, conclusions are drawn in Section~\ref{sec:conclusion}. Throughout this paper, we adopt the standard Sobolev spaces as well as their corresponding norms, use c.c. to denote complex conjugate of the term in front of it,  and adopt $A \lesssim B$ to represent that there exists a generic constant $C > 0$ independent of $\varepsilon$, $h$ and $\tau$, such that $|A| \le CB$.


\section{Multiscale decomposition by amplitude and /or frequency}\label{sec:multi}

In this section, we present a multiscale decomposition of the KGS system \eqref{KGSndim}-\eqref{KGSini}. Similar to the asymptotic behavior \eqref{ansatzRd}, we make the following ansatz \cite{bao2017uniformly}
\begin{align}
\phi(\bx,t) = e^{i\mu t / \varepsilon^2} v(\bx,t) + e^{-i\mu t/ \varepsilon^2} \overline{v(\bx,t)} + r(\bx,t), \quad \bx \in \mathbb{R}^d, \quad t \ge 0, \label{ansatz}
\end{align}
and plug it into the nonlinear Klein-Gordon equation and the initial data \eqref{KGSini}. After some computation, we obtain the following equation
\begin{align}
    0 = & \ \varepsilon^2 \partial_{tt} \phi(\bx,t) - \Delta \phi(\bx,t) + \frac{\mu^2}{\varepsilon^2}\phi(\bx,t) - \lambda |\psi(\bx,t)|^2 \nonumber \\
    = &\ e^{i\mu t / \varepsilon^2} \left[ \varepsilon^2 \partial_{tt} v(\bx,t) + 2i \mu \partial_t v(\bx,t) - \Delta v(\bx,t) \right] \nonumber \\
     &\ +  e^{-i\mu t / \varepsilon^2} \left[ \varepsilon^2  \overline{\partial_{tt} v(\bx,t)} - 2i \mu  \overline{ \partial_t v(\bx,t) }-  \overline{\Delta v(\bx,t)}\right]  \nonumber \\
    &\ + \varepsilon^2 \partial_{tt} r(\bx,t) - \Delta r(\bx,t) + \frac{\mu^2}{\varepsilon^2} r(\bx,t) - \lambda |\psi(\bx,t)|^2, \label{decom}
\end{align}
and the initial conditions
\begin{equation}
\label{deci}
 \left\{   \begin{aligned}
        & \phi_0(\bx) = v(\bx,0) + \overline{v(\bx,0)} + r(\bx,0) , \qquad \bx \in \mathbb{R}^d, \\
        & \frac{1}{\varepsilon^2} \phi_1(\bx) = \frac{i\mu}{\varepsilon^2} \left( v(\bx,0) - \overline{v(\bx,0)} \right) + \partial_t v(\bx,0) + \overline{\partial_t v(\bx,0)} + \partial_t r(\bx,0). 
    \end{aligned}\right.
\end{equation}

\noindent In order to make the remainder term $r$ small, we request
\begin{align}
r(\bx,0) = 0  \quad \Longrightarrow \quad v(\bx,0) + \overline{v(\bx,0)} = \phi_0(\bx) =O(1), \quad \bx \in \mathbb{R}^d. \label{scale1}
\end{align}

\noindent Similar to the scale separation in \cite{bao2012uniformly}, we let $\varepsilon \to 0$ in \eqref{deci} and obtain
\begin{align}
i \mu \left( v(\bx,0) - \overline{v(\bx,0)} \right) = \phi_1 (\bx), \qquad \partial_t r(\bx,0) + \partial_t v(\bx,0) + \overline{\partial_t v(\bx,0)} = 0. \label{scale2}
\end{align}

\noindent Combining \eqref{scale1} and \eqref{scale2}, we get
\begin{align}
 &   v(\bx,0) = \frac{1}{2} \left[ \phi_0(\bx) - \frac{i}{\mu} \phi_1(\bx) \right] :=v_0(\bx), \qquad \bx \in \mathbb{R}^d, \label{vi}\\
 & r(\bx,0) = 0, \qquad \partial_t r(\bx,0) = -\partial_t v(\bx,0) - \overline{\partial_t v(\bx,0)}.\label{ri}
\end{align}

By adopting a multiscale decomposition by amplitude with $O(1)$ and frequency with $\varepsilon^{-2}$ in \eqref{decom}, we request that $v : = v(\bx,t)$ satisfies the following Schr\"{o}dinger equation with wave operator as
\begin{equation}
\label{v}
\left\{\begin{aligned}
& \varepsilon^2 \partial_{tt} v(\bx,t) + 2i \mu \partial_t v(\bx,t) - \Delta v(\bx,t)=0, \quad \bx \in \mathbb{R}
^d, \quad t >0, \\
& v(\bx,0) = v_0 (\bx) = \frac{1}{2} \left[ \phi_0(\bx) - \frac{i}{\mu} \phi_1(\bx) \right], \quad \partial_t v(\bx,0) = \gamma (\bx),
\end{aligned}\right.
\end{equation}
and the remainder $r:=r(\bx,t)$ satisfies the nonlinear Klein-Gordon equation with source terms and small initial data
\begin{equation}
\label{r}
\left\{\begin{aligned}
&  \varepsilon^2 \partial_{tt}r(\bx,t) -\Delta r(\bx,t) +\frac{\mu^2}{\varepsilon^2} r(\bx,t) -\lambda |\psi(\bx,t)|^2=0,\quad \bx \in \mathbb{R}^d, \quad t >0, \\
    & r(\bx,0) = 0,\quad \partial_t r(\bx,0) = -\gamma(\bx) - \overline{\gamma(\bx)}, 
    \end{aligned}\right.
\end{equation}
where $\gamma(\bx)$ is a given complex-valued function to be specified later. It is easy to check that, if $v$ and $r$ are solutions of \eqref{v} and \eqref{r}, respectively, then $\phi$ in \eqref{ansatz} is a solution of the nonlinear Klein-Gordon equation. In this case, the nonlinear Schr\"odinger equation in \eqref{KGSndim} for $\psi := \psi(\bx, t)$ could be rewritten as
\begin{equation}
\label{psi}
\left\{
\begin{aligned}
   & i \partial_t \psi + \Delta \psi + \lambda \left[e^{i\mu t / \varepsilon^2} v(\bx,t) + {\rm c.c.} + r(\bx,t) \right] \psi = 0, \quad t > 0, \\
   & \psi(\bx,0) = \psi_0(\bx),\quad \bx \in \mathbb{R}^d.    
   \end{aligned}\right.
\end{equation}

Different $\gamma(\bx)$ can be used in the multiscale decomposition by frequency \eqref{v} and \eqref{r}.
In the reference \cite{bao2017uniformly}, $\gamma(\bx)$ is taken as the well-prepared initial data, i.e. it is obtained by 
setting $t=0$ and $\varepsilon \to 0$ in \eqref{v}, as
\begin{align}
\gamma(\bx)  = -\frac{i}{2\mu}\Delta v(\bx,0) = -\frac{i}{2\mu}\Delta v_0(\bx) :=v_1(\bx), \quad \bx \in \mathbb{R}^d.\label{wellprepared}
\end{align}
Under the well-prepared initial data for \eqref{v}, formally one can show that
$v(\bx,t)$, $\partial_t v(\bx,t)$ and $\partial_{tt} v(\bx,t)$ are uniformly bounded with respect to $\varepsilon\in(0,1]$ \cite{bao2017uniformly}. Based on this choice of $\gamma(\bx)$ for the multiscale decomposition by frequency \eqref{v} and \eqref{r}, a multiscale time integrator Fourier pseudospectral (MTI-FP) method was presented and analyzed in the literature \cite{bao2017uniformly}. 

On the other hand, we can take the most simplified initial condition as $\gamma(\bx)\equiv0$, which immediately implies $\partial_t v(\bx,0)=\partial_t r(\bx,0)\equiv 0$ in \eqref{v} and \eqref{r}, i.e., the initial conditions for $v$ and $r$ are taken as homogeneous except $v(\bx,0)$.
Based on this most simplified initial conditions, we will present a new and simplified MTI-FP method for the KGS system \eqref{KGSndim}-\eqref{KGSini}  and 
establish its uniform first-order accuracy in time with respect to $\varepsilon\in(0,1]$. Compared to the MTI-FP method in the literature \cite{bao2017uniformly}, 
due to the adopted homogeneous initial conditions, the proposed MTI-FP method is significantly simplified and thus the computational cost is greatly reduced in practical computation, and the regularity requirement is also weaker to obtain the same first-order error bound in time.


\section{A uniformly accurate MTI-FP method}\label{sec:MTI}
In this section, we present a multiscale time integrator Fourier pseudospectral (MTI-FP) method for the KGS system \eqref{KGSndim}-\eqref{KGSini} based on the multiscale decomposition by frequency \eqref{v}-\eqref{psi} with $\gamma(\bx) \equiv 0$. 

For simplicity of notation and without loss of generality, we only present the MTI-FP method for the KGS system \eqref{KGSndim}-\eqref{KGSini} with $\mu = \lambda = 1$ and in one dimension (1D), and  generalization to higher dimensions is straightforward by tensor product. As in the literatures \cite{bao2017uniformly,zhang2010finite}, 
the KGS system \eqref{KGSndim}-\eqref{KGSini} with $d = 1$ is usually truncated onto a bounded interval 	$\Omega=(a,b)$ ($|a|$ and $b$ are taken large enough such that the truncation error is negligible) with periodic boundary conditions
\begin{equation}
\label{psi1d}
\left\{
\begin{aligned}
 & i \partial_t \psi(x,t) + \partial_{xx} \psi(x,t) + \phi(x,t) \psi(x,t) = 0, \quad x \in \Omega, \quad t > 0,  \\
 & \psi(x,0) = \psi_0(x), \quad x \in \overline{\Omega}, \\
 & \psi(a,t) = \psi(b,t), \quad \partial_x \psi(a,t) = \partial_x \psi(b,t),\quad t \ge 0, 
 \end{aligned}\right.
 \end{equation}
and 
\begin{equation}
\label{phi1d}
\left\{\begin{aligned}
& \varepsilon^2 \partial_{tt} \phi(x,t) - \partial_{xx} \phi(x,t) + \displaystyle \frac{1}{\varepsilon^2} \phi(x,t) - | \psi(x,t)|^2 =0, \quad x \in \Omega, \quad t > 0, \\
& \phi(x,0) = \phi_0(x), \quad \partial_t \phi(x,0) = \frac{1}{\varepsilon^2} \phi_1(x), \quad x \in \overline{\Omega}, \\
& \phi(a,t) = \phi(b,t),\quad \partial_x \phi(a,t) = \partial_x \phi(b,t), \quad t \ge 0.
\end{aligned}\right.
\end{equation}

\subsection{A multiscale decomposition via simplified transmission conditions}
Let $\tau=\Delta t>0$ be the time step size, and denote time levels by $t_n = n \tau$ for $n = 0,1,...$.
We first present a multiscale decomposition by frequency for solutions of \eqref{psi1d}-\eqref{phi1d} on the time
interval $[t_n,t_{n+1}]$ with given initial data at $t=t_n$ as
\begin{equation}
\label{I_n2}
\left\{\begin{aligned}
& \psi(x,t_n) = \psi^n_0(x) = O(1), \qquad x \in \overline{\Omega},\\
& \phi(x,t_n) = \phi^n_0(x) = O(1),\qquad \partial_t \phi(x,t_n) = \frac{1}{\varepsilon^2} \phi^n_1(x) = O\left(\frac{1}{\varepsilon^2}\right).
\end{aligned}\right.
\end{equation}

Similar to the derivation in Section~\ref{sec:multi}, we take an ansatz to the solution $\phi(x,t):= \phi(x,t_n + s)$ of \eqref{phi1d} on the interval $[t_n,t_{n+1}]$ with the initial data in \eqref{I_n2} as
\begin{equation}
    \phi(x,t_n+s) = e^{i s / \varepsilon^2} v^n(x,s) + e^{-i s / \varepsilon^2} \overline{v^n(x,s)} + r^n(x,s), \quad x \in \overline{\Omega},\quad 0 \le s \le \tau. \label{ansatz1d}
\end{equation}

Then a multiscale decomposition by frequency with simplified transmission conditions for the KGS system \eqref{psi1d}-\eqref{phi1d} over the time interval $[t_n , t_{n+1}]$ is: $v^n :=v^n(x,s)$ satisfies the following Schr\"{o}dinger equation with wave operator under simplified transmission conditions as
\begin{equation}
\label{v1d}
\left\{\begin{aligned}
& \varepsilon^2 \partial_{ss} v^n(x,s) + 2i \partial_s v^n(x,s) - \partial_{xx} v^n(x,s)=0, \quad x \in \Omega, \quad 0 \le s \le \tau,  \\
& v^n(x,0) = \frac{1}{2} \left[ \phi_0^n(x) - i \phi_1^n(x) \right], \quad \partial_s v^n(x,0) = 0,\quad x \in \overline{\Omega}, \\
& v^n(a,s) = v^n(b,s),\quad \partial_x v^n(a,s) = \partial_x v^n(b,s), \quad 0 \le s \le \tau,
\end{aligned}\right.
\end{equation}
and the remainder $r^n :=r^n(x,s)$ satisfies the following nonlinear Klein-Gordon equation with source terms and homogeneous initial data as
\begin{equation}
\label{r1d}
\left\{\begin{aligned}
&  \varepsilon^2 \partial_{ss}r^n(x,s) -\partial_{xx} r^n(x,s) +\frac{1}{\varepsilon^2} r^n(x,s) - |\psi^n(x, s)|^2=0, \quad  0 \le s \le \tau, \\
    & r^n(x,0) = 0,\quad \partial_s r^n(x,0) = 0, \quad x \in \overline{\Omega},  \\
    & r^n(a,s) = r^n(b,s), \quad \partial_x r^n(a,s) = \partial_x r^n(b,s), \quad 0 \le s \le \tau,
\end{aligned}\right.
\end{equation}
where $\psi^n(x,s) : = \psi(x,t_n + s)$ satisfies the corresponding Schr\"odinger equation 
\begin{equation}
\label{psi1d(D)}
\left\{\begin{aligned}
 & i \partial_s \psi^n(x,s) + \partial_{xx} \psi^n(x, s) +  \left[ e^{i s / \varepsilon^2} v^n(x,s) + {\rm c.c.} + r^n(x,s) \right] \psi^n(x,s) = 0,  \\
 & \psi^n(x,0) = \psi_0^n(x), \quad x \in \overline{\Omega}, \\
 & \psi^n(a, s) = \psi^n(b, s), \quad \partial_x \psi^n(a, s) = \partial_x \psi^n(b,s),\quad 0 \le s \le \tau.
 \end{aligned}\right.
\end{equation}

\subsection{ A multiscale time integrator Fourier pseudospectral method}
Denote the mesh size as $h : = (b-a)/N$ with $N$ a positive even integer and grid points as $x_j := a + jh$ for $j = 0,1,...,N$. Define
\begin{align*}
   & X_N := \text{span} \left\{ e^{i\mu_l(x - a)} \ | \ l = \displaystyle -\frac{N}{2},..., \frac{N}{2} - 1 \right\} \quad \text{with} \ \ \mu_l : = \displaystyle \frac{2\pi l}{b-a}, \\
   & Y_N := \left\{ \mathbf{v} = (v_0,v_1,...,v_N)^T \in \mathbb{C}^{N+1} \ | \ v_0 = v_N  \right\} \quad \text{with} \ \ \left\|   \mathbf{v}
   \right\|_{l^2} : = \left( h \sum_{j = 0}^{N-1} | v_j|^2 \right)^{1/2}. 
\end{align*}
For a periodic function $v(x)$ on $\overline{\Omega}$ and a vector $\mathbf{v} \in Y_N$, let $P_N:L^2(\Omega) \rightarrow X_N$ denote the standard $L^2 $-projection operator, and let $I_N: C(\Omega) \rightarrow X_N$ or $Y_N \rightarrow X_N$ be the trigonometric interpolation operator, i.e.,
\begin{align*}
    (P_Nv)(x) = \sum_{l = -N / 2}^{N / 2 - 1} \widehat{v}_l e^{i \mu_l ( x - a)}, \quad (I_N \mathbf{v})(x) = \sum_{l = -N / 2}^{N / 2 - 1} \widetilde{\mathbf{v}}_l e^{i \mu_l (x-a)}, \quad a \le x \le b, 
\end{align*}
where $\widehat{v}_l$ and $\widetilde{\mathbf{v}}_l$ are the Fourier and discrete Fourier transform coefficients of the periodic function $v(x)$ and the vector $\mathbf{v}$, respectively, defined as 
\begin{align*}
    \widehat{v}_l = \displaystyle \frac{1}{b-a} \int_a^b v(x) e^{-i \mu_l (x - a)} dx, \quad \widetilde{\mathbf{v}}_l = \frac{1}{N} \sum_{j = 0}^{N-1} v_j e^{-i \mu_l (x_j - a)},
\end{align*}
with $v_j:=v(x_j)$ for $j=0,1,\ldots,N$.

To solve the equations \eqref{v1d}-\eqref{psi1d(D)} numerically, we apply exponential wave integrators in time \cite{bao2017uniformly} and Fourier spectral/pseudospectral method \cite{cai2023uniformly,shen2011spectral} in space. In space, we first apply the Fourier spectral method for discretizing \eqref{v1d}-\eqref{psi1d(D)}, i.e., to find $v^n_N(x,s)$, $r^n_N(x,s)$ and $\psi^n_N(x,s) \in X_N$ as
\begin{equation}
\label{FS}
\begin{aligned}
&v_N^n(x,s) = \sum_{l=-N/2}^{N/2-1} \widehat{(v^n_N)}_l(s) e^{i\mu_l (x-a)}, \quad r_N^n(x,s) = \sum_{l=-N/2}^{N/2-1} \widehat{(r^n_N)}_l(s) e^{i\mu_l (x-a)}, \\
& \psi_N^n(x,s) = \sum_{l=-N/2}^{N/2-1} \widehat{(\psi^n_N)}_l(s) e^{i\mu_l (x-a)}, \quad x \in \overline{\Omega}, \quad 0 \le s \le \tau,
\end{aligned}
\end{equation}
such that they are the solutions of the multiscale decomposition \eqref{v1d}-\eqref{r1d} and the Schr\"odinger equation \eqref{psi1d(D)}, respectively. Pluging \eqref{FS} into equations \eqref{v1d}-\eqref{psi1d(D)} and noticing the orthogonality of $e^{i\mu_l (x-a)}$ for $l=-\frac{N}{2},\ldots,\frac{N}{2}-1$, we get
 \begin{align*}
    & \varepsilon^2 \widehat{(v^n_N)}_l''(s) +  2i \widehat{(v^n_N)}_l'(s) + \mu_l^2 \widehat{(v^n_N)}_l(s)=0, \qquad 0 \le s \le \tau, \\
    & \varepsilon^2 \widehat{(r^n_N)}_l''(s) + ( \mu_l^2 + \frac{1}{\varepsilon^2} ) \widehat{(r^n_N)}_l(s) - \widehat{( |\psi^n_N|^2)}_l(s) =0, \\
    & i \widehat{(\psi^n_N)}_l'(s) - \mu_l^2 \widehat{(\psi^n_N)}_l(s) +  e^{i s / \varepsilon^2} \widehat{(v^n_N \psi^n_N)}_l(s) +  e^{-i s / \varepsilon^2} \widehat{(\overline{v^n_N} \psi^n_N)}_l(s) + \widehat{(r^n_N \psi^n_N)}_l(s) =0.
\end{align*}
By Duhamel's principle, we apply the initial data in \eqref{v1d} and \eqref{r1d} and obtain
\begin{align}
\widehat{(v^n_N)}_l(s) = & \ a_l(s) \widehat{(v^n_N)}_l(0), \qquad 0 \le s \le \tau, \label{Sv}\\
\widehat{(r^n_N)}_l(s) = & \  \int_0^s \frac{\sin(\omega_l (s - \theta))}{\varepsilon^2 \omega_l} \widehat{(|\psi^n_N|^2)}_l(\theta) d\theta, \label{Sr}\\
\widehat{(\psi^n_N)}_l(s) = &\  e^{-i \mu_l^2 s }\widehat{(\psi^n_N)}_l(0) + i e^{-i\mu_l^2 s} \int_0^s e^{i(\mu_l^2 + \frac{1}{\varepsilon^2}) \theta} \widehat{(v^n_N \psi^n_N)}_l(\theta) d\theta  \nonumber \\
& + i e^{-i\mu_l^2 s} \int_0^s e^{i(\mu_l^2 - \frac{1}{\varepsilon^2}) \theta} \widehat{( \overline{v^n_N} \psi^n_N)}_l(\theta) d\theta \nonumber \\
& + i \int_0^s e^{i\mu_l^2 (\theta - s)} \widehat{(r^n_N \psi^n_N)}_l(\theta) d\theta, \label{Sp}
\end{align}
where $\omega_l = \frac{1}{\varepsilon^2}\sqrt{1 + \varepsilon^2 \mu_l^2}$ for $l=-\frac{N}{2},\ldots, \frac{N}{2}-1$, and 
\begin{align}
\lambda_l^{\pm} := -  \displaystyle\frac{1}{\varepsilon^2}\left[
  1 \pm \sqrt{1 + \varepsilon^2 \mu_l^2}\right], \quad a_l(s) := \displaystyle \frac{\lambda_l^+ e^{is \lambda_l^-} - \lambda_l^-e^{is \lambda_l^+}}{\lambda_l^+ - \lambda_l^-}, \quad 0 \le s \le \tau,
    \label{Coeffa}
\end{align}
with
\begin{align}
    v^n_N(x,0) = \frac{1}{2}P_N \left( \phi_0^n - i \phi_1^n \right)(x), \quad \psi_N^n (x,0) = ( P_N \psi_0^n )(x). \label{FSinitial}
\end{align}
Then, differentiating \eqref{Sv}-\eqref{Sr} with respect to $s$, we have for $0 \le s \le \tau$
\begin{align}
\widehat{(v^n_N)}_l'(s) = & a_l'(s) \widehat{(v^n_N)}_l(0), \quad \widehat{(r^n_N)}_l'(s) =  \int_0^s \frac{\cos(\omega_l (s - \theta))}{\varepsilon^2} \widehat{(|\psi^n_N|^2)}_l(\theta) d\theta, \label{Sdvdr}
\end{align}
where
\begin{align}
    a_l'(s) = i \lambda_l^+ \lambda_l^- \displaystyle \frac{e^{is \lambda_l^-} - e^{is \lambda_l^+}}{\lambda_l^+ - \lambda_l^-}, \qquad l = - \frac{N}{2},\ldots,\frac{N}{2}-1.
    \label{Coeffda}
\end{align}

To approximate the integrals in \eqref{Sr}-\eqref{Sp} and \eqref{Sdvdr}, we apply the Gautschi's type quadrature \cite{gautschi1961numerical}, i.e.,
\begin{align*}
&\int_0^s \frac{\sin(\omega_l(s - \theta))}{\varepsilon^2 \omega_l} \widehat{(\left| \psi_N^n \right|^2)}_l(\theta ) d \theta \approx \int_0^s \frac{\sin(\omega_l(s - \theta))}{\varepsilon^2 \omega_l} \left[  \widehat{(\left| \psi_N^n \right|^2)}_l(0) + \theta \widehat{(\left| \psi_N^n \right|^2)}_l'(0) \right] d \theta, \\
& \int_0^s e^{i ( \mu_l^2 + \frac{1}{\varepsilon^2}) \theta} \widehat{( v_N^n \psi^n_N)}_l(\theta ) d \theta \approx \int_0^s e^{i ( \mu_l^2 + \frac{1}{\varepsilon^2}) \theta} \left[ \widehat{( v_N^n \psi^n_N)}_l(0 ) + \theta \widehat{( v_N^n \psi^n_N)}_l'(0) \right] d \theta,\\
& \int_0^s e^{i\mu_l^2 (\theta - s)} \widehat{(r^n_N \psi^n_N)}_l(\theta) d\theta \approx \int_0^s e^{i\mu_l^2 (\theta - s)} \left[ \widehat{(r^n_N \psi^n_N)}_l(0) + \theta \widehat{(r^n_N \psi^n_N)}_l'(0) \right] d\theta.
\end{align*}
Thus, by taking $s = \tau$ in \eqref{Sv}-\eqref{Sp} and \eqref{Sdvdr}, and combining \eqref{Coeffa}-\eqref{FSinitial} and \eqref{Coeffda}, we obtain the following multiscale time integrator Fourier spectral (MTI-FS) method
\begin{equation}
\label{FS_scheme}
\left\{\begin{aligned}
\widehat{(v_N^n)}_l(\tau) =&\ a_l(\tau) \widehat{(v_N^n)}_l(0), \quad \widehat{(v_N^n)}_l'(\tau) = a_l'(\tau) \widehat{(v_N^n)}_l(0), \\
 \widehat{(r_N^n)}_l(\tau) \approx &\ p_l(\tau) \widehat{(\left| \psi_N^n \right|^2)}_l (0) + q_l(\tau) \widehat{(\left| \psi_N^n \right|^2)}_l'(0), \\
 \widehat{(r_N^n)}_l'(\tau) \approx &\ p_l'(\tau) \widehat{(\left| \psi_N^n \right|^2)}_l (0) + q_l'(\tau) \widehat{(\left| \psi_N^n \right|^2)}_l'(0),\\
 \widehat{(\psi^n_N)}_l(\tau) \approx &\ e^{-i \mu_l^2 \tau }\widehat{(\psi^n_N)}_l(0) + c_l^+(\tau) \widehat{(v^n_N \psi^n_N)}_l(0) +  d_l^+(\tau) \widehat{(v^n_N \psi^n_N)}_l'(0)  \\
& + c_l^-(\tau) \widehat{(\overline{v^n_N} \psi^n_N)}_l(0) +  d_l^-(\tau) \widehat{(\overline{v^n_N} \psi^n_N)}_l'(0),
\end{aligned}\right.
\end{equation}
where
\begin{align}
\begin{cases}
p_l(\tau) = \displaystyle \int_0^{\tau} \frac{\sin(\omega_l(\tau - \theta))}{\varepsilon^2 \omega_l} d\theta, \quad q_l(\tau) = \int_0^{\tau}  \theta \frac{\sin(\omega_l(\tau - \theta))}{\varepsilon^2 \omega_l} d\theta, \\
p_l'(\tau) = \displaystyle \int_0^{\tau}  \frac{\cos(\omega_l(\tau - \theta))}{\varepsilon^2 } d\theta, \quad q_l'(\tau) = \int_0^{\tau}  \theta \frac{\cos(\omega_l(\tau - \theta))}{\varepsilon^2 } d\theta, \\
c_l^{\pm}(\tau) =  \displaystyle i e^{-i\mu_l^2 \tau} \int_0^{\tau} e^{i \left( \mu_l^2 \pm \frac{1}{\varepsilon^2} \right) \theta} d\theta, \\
d_l^{\pm}(\tau) =  \displaystyle i e^{-i\mu_l^2 \tau} \int_0^{\tau} \theta e^{i \left( \mu_l^2 \pm \frac{1}{\varepsilon^2} \right) \theta} d\theta.
\end{cases}
\end{align}
The time derivatives in \eqref{FS_scheme} are given by
\begin{equation}
\label{FS_Timed}
\left\{\begin{aligned}
& \widehat{(\psi_N^n)}_l'(0) \approx -i \frac{\sin(\mu_l^2 \tau)}{\tau} \widehat{(\psi_N^n)}_l(0) + i \widehat{(v_N^n \psi_N^n)}_l(0) + i \widehat{(\overline{v_N^n} \psi_N^n)}_l(0),\\
& \widehat{(\left| \psi_N^n \right|^2)}_l'(0)  = 2\widehat{({\rm Re} \{ \overline{\psi_N^n}\partial_s\psi_N^n\})}_l (0), \qquad l = -\frac{N}{2},\ldots,\frac{N}{2}-1, \\
& \widehat{(v_N^n \psi_N^n)}_l'(0) = \widehat{(v_N^n \partial_s \psi_N^n)}_l(0) + \widehat{(\psi_N^n \partial_s v_N^n)}_l(0) = \widehat{(v_N^n \partial_s \psi_N^n)}_l(0),
\end{aligned}\right.
\end{equation}
where ${\rm Re}(f)$ represents the real part of the function $f$.

In practical computation, the integrals in Fourier coefficients are usually approximated by numerical quadratures \cite{gautschi1961numerical,shen2011spectral}. Let $\phi_j^n$, $\dot{\phi}_j^n$, $\psi_j^n$ and $\dot{\psi}_j^n$ be approximations of $\phi(x_j,t_n)$, $\partial_t \phi(x_j,t_n)$, $\psi(x_j,t_n)$ and $\partial_t \psi(x_j,t_n)$, respectively, and let $v^{n,1}_j$, $\dot{v}^{n,1}_j$, $r^{n,1}_j$,  $\dot{r}^{n,1}_j$ be approximations of $v^n(x_j,\tau)$, $\partial_s v^n(x_j,\tau)$, $r^n(x_j,\tau)$ and $\partial_s r^n(x_j,\tau)$, respectively, for $n = 0,1,\ldots$ and $j = 0,1,\ldots,N$. Denote $\Phi^n=(\phi_0^n,\phi_1^n,\ldots,\phi_N^n)^T\in Y_N$, $\dot{\Phi}^n=(\dot{\phi}_0^n,\dot{\phi}_1^n,\ldots,
\dot{\phi}_N^n)^T\in Y_N$, $\Psi^n=(\psi_0^n,\psi_1^n,\ldots,\psi_N^n)^T\in Y_N$, $\dot{\Psi}^n=(\dot{\psi}_0^n,\dot{\psi}_1^n,\ldots,
\dot{\psi}_N^n)^T\in Y_N$, $\mathbf{v}^{n,1}=(v_0^{n,1},v_1^{n,1},\ldots,v_N^{n,1})^T\in Y_N$, $\dot{\mathbf{v}}^{n,1}=(\dot{v}_0^{n,1},\dot{v}_1^{n,1},\ldots,\dot{v}_N^{n,1})^T\in Y_N$, $\mathbf{r}^{n,1}=(r_0^{n,1},r_1^{n,1},\ldots,r_N^{n,1})^T\in Y_N$ and $\dot{\mathbf{r}}^{n,1}=(\dot{r}_0^{n,1},$ $\dot{r}_1^{n,1},\ldots,\dot{r}_N^{n,1})^T\in Y_N$.
Then a multiscale time integrator Fourier pseudospectral (MTI-FP) method for the KGS system \eqref{psi1d}-\eqref{phi1d} reads as
\begin{equation}
\label{MTIFP}
\left\{\begin{aligned}
&\phi_j^{n+1} = e^{i\tau / \varepsilon^2} v^{n,1}_j + e^{-i\tau / \varepsilon^2} \overline{v^{n,1}_j} + r_j^{n,1}, \qquad j = 0,1,...,N, \qquad n \ge 0, \\
&\dot{\phi}_j^{n+1} = e^{i\tau / \varepsilon^2} \left( \dot{v}_j^{n,1} + \displaystyle \frac{i}{\varepsilon^2} v^{n,1}_j \right) + e^{-i\tau / \varepsilon^2} \left( \overline{\dot{v}_j^{n,1}} - \displaystyle \frac{i}{\varepsilon^2} \overline{v^{n,1}_j} \right) + \dot{r}_j^{n,1}, \\
& \psi_j^{n+1} = \sum_{l = -N/2}^{N/2 - 1} \widetilde{(\Psi^{n+1})}_l e^{i\mu_l(x_j-a)},
\end{aligned}\right.
\end{equation}
where
\begin{subequations}
\label{MTIFPvr}
\begin{align}
& v_j^{n,1} = \sum_{l = -N / 2}^{N/2-1} \widetilde{(\mathbf{v}^{n,1})}_l e^{i \mu_l (x_j - a)}, \quad & \dot{v}_j^{n,1} = \sum_{l = -N / 2}^{N/2 - 1} \widetilde{(\dot{\mathbf{v}}^{n,1})}_l e^{i \mu_l (x_j -a)}, \\
&  r_j^{n,1} = \sum_{l = -N / 2}^{N/2-1} \widetilde{(\mathbf{r}^{n,1})}_l e^{i \mu_l (x_j - a)}, \quad & \dot{r}_j^{n,1} = \sum_{l = -N / 2}^{N/2 - 1} \widetilde{(\dot{\mathbf{r}}^{n,1})}_l e^{i \mu_l (x_j -a)},
\end{align}
\end{subequations}
with the Fourier coefficients given by
\begin{subequations}\label{MTIFP_Covrp}
\begin{align}
\widetilde{(\mathbf{v}^{n,1})}_l =&\ a_l(\tau) \widetilde{(\mathbf{v}^{n,0})}_l, \quad   
\widetilde{(\dot{\mathbf{v}}^{n,1})}_l = a_l'(\tau) \widetilde{(\mathbf{v}^{n,0})}_l , \label{MTIFP.Cv}\\
\widetilde{(\mathbf{r}^{n,1})}_l =&\ p_l(\tau)  \widetilde{(\left| \Psi^n \right|^2)}_l   + 2 q_l(\tau)  \widetilde{( {\rm Re}\{\overline{\Psi^n} \dot{\Psi}^n\})}_l , \quad l = \displaystyle -\frac{N}{2},..., \frac{N}{2}-1, \label{MTIFP.Cr} \\
\widetilde{(\dot{\mathbf{r}}^{n,1})}_l  = &\ p_l'(\tau)  \widetilde{(\left| \Psi^n \right|^2)}_l   + 2 q_l'(\tau)  \widetilde{( {\rm Re}\{\overline{\Psi^n} \dot{\Psi}^n\})}_l, \\
 \widetilde{(\Psi^{n+1})}_l =  &\ e^{-i \mu_l^2 \tau }\widetilde{(\Psi^n)}_l + c_l^+(\tau) \widetilde{(\mathbf{v}^{n,0} \Psi^n)}_l +  d_l^+(\tau) \widetilde{(\mathbf{v}^{n,0} \dot{\Psi}^n)}_l\nonumber \\
& + c_l^-(\tau) \widetilde{(\overline{\mathbf{v}^{n,0}} \Psi^n)}_l +  d_l^-(\tau) \widetilde{(\overline{\mathbf{v}^{n,0}} \dot{\Psi}^n)}_l, \label{MTIFP.CPsi}
\end{align}
\end{subequations}
and
\begin{subequations}\label{MTIFP_Ini}
\begin{align}
&\mathbf{v}^{n,0}=(v^{n,0}_0,v^{n,0}_1,\ldots,v^{n,0}_N)^T\in Y_N,\quad \mathrm{with} \quad v^{n,0}_j = \displaystyle \frac{1}{2} \left(\phi_j^n - i\varepsilon^2 \dot{\phi}_j^n  \right),\label{MTIIni1} \\
& \widetilde{(\dot{\Psi}^n)}_l = - i \frac{ \sin(\mu_l^2 \tau)}{\tau} \widetilde{(\Psi^n)}_l + i \widetilde{(\Phi^n \Psi^n)}_l, \quad l = -\frac{N}{2},\ldots,\frac{N}{2}-1,\quad n\ge0 . \label{MTIIni2}
\end{align}
\end{subequations}
The initial data is given as
\begin{equation}\label{init00}
\phi^0_j=\phi_0(x_j), \quad \dot{\phi}^0_j=\frac{1}{\varepsilon^2}\phi_1(x_j),\quad \psi_j^0 = \psi_0(x_j), \qquad j=0,1,\ldots, N.
\end{equation}

The above MTI-FP method for \eqref{psi1d}-\eqref{phi1d} is implemented by the fast Fourier transform (FFT). Thus, this method is fully explicit, accurate and very efficient, and its memory cost is $O(N)$ and the computational cost per time step is $O(N  \text{log} N)$.


\section{A uniformly accurate error bound}\label{sec:unierr}
In this section, we shall establish two independent error bounds for the MTI-FP method \eqref{MTIFP}-\eqref{init00}, which immediately implies a uniformly accurate error bound with respect to $\varepsilon \in (0,1]$.

Let $T^*$ be the maximum existence time of the solutions $\psi(x,t)$ and $\phi(x,t)$ to the KGS system \eqref{psi1d}-\eqref{phi1d} and take $0 < T < T^*$ as some fixed time. We make the following assumption (A) of $\psi(x,t)$ and $\phi(x,t)$, i.e., there exists an integer $m_0 \ge 5$ such that
\begin{equation*}\text{(A)} 
\begin{aligned} \qquad & \psi \in C \left([0,T]; H^{m_0}_{\rm per} (\Omega) \right), \quad && \left\| \psi \right\|_{L^{\infty} \left( [0,T]; H^{m_0} \right)} \lesssim 1,  \\
 & \phi \in C^1 \left([0,T]; H^{m_0}_{\rm per} (\Omega) \right), \quad && \left\| \phi \right\|_{L^{\infty} \left( [0,T]; H^{m_0} \right)} + \varepsilon^2 \left\| \partial_t \phi \right\|_{L^{\infty} \left( [0,T]; H^{m_0} \right)} \lesssim 1,
\end{aligned}
\end{equation*}
where $H_{\rm per}^m(\Omega) = \left\{ u(x) \in H^m(\Omega) \ | \ u^{(k)}(a) = u^{(k)} (b), \ k = 0,1,\ldots, m-1  \right\}$.

\subsection{Main results}
In this subsection, we present our main error estimates under assumption (A).
We introduce the following notations
\begin{align*}
& C_\psi = \max_{0<\varepsilon \le 1} \left\{ \left\| \psi \right\|_{L^{\infty} \left( [0,T]; H^{m_0} \right)}\right\}, \\ 
& C_\phi = \max_{0<\varepsilon \le 1} \left\{ \left\| \phi \right\|_{L^{\infty} \left( [0,T]; H^{m_0} \right)}, \varepsilon^2 \left\| \partial_t \phi \right\|_{L^{\infty} \left( [0,T]; H^{m_0} \right)} \right\},
\end{align*}
and define the error functions as
\begin{equation}
\begin{aligned}
  &  e^n_\psi(x) := \psi(x,t_n) - (I_N\Psi^n)(x), \qquad x \in \overline{\Omega
  }, \qquad n \ge 0, \\
  &  e^n_\phi(x) := \phi(x,t_n) - (I_N\Phi^n)(x), \qquad \dot{e}^n_\phi(x) := \partial_t \phi(x,t_n) - (I_N\dot{\Phi}^n)(x),
\end{aligned}
\end{equation}
where $\Psi^n$, $\Phi^n$ and $\dot{\Phi}^n$ are the numerical solutions obtained from the MTI-FP method \eqref{MTIFP}-\eqref{init00}. Then we have the following error estimates.

\begin{Theorem} Under the assumption (A), there exist two constants $0 < h_0 < 1$ and $0 < \tau_0 < 1$ sufficiently small and independent of $\varepsilon$ such that for any $0 < \varepsilon \le 1$, when $0 < h \le h_0$ and $0 < \tau \le \tau_0$, we have for $0 \le n \leq \frac{T}{\tau}$
\begin{align}
    & \left\| e^n_\psi \right\|_{H^1} + \left\| e^n_\phi \right\|_{H^1} + \varepsilon^2 \left\| \dot{e}^n_\phi \right\|_{H^1} \lesssim h^{m_0 -1 } + \displaystyle \frac{\tau^2}{\varepsilon^2}, \label{thm1err1}\\ 
    & \left\| e^n_\psi \right\|_{H^1}  + \left\| e^n_\phi \right\|_{H^1} + \varepsilon^2 \left\| \dot{e}^n_\phi \right\|_{H^1} \lesssim h^{m_0 -1} +  \varepsilon^2, \label{thm1err2} \\
    & \left\| I_N \Psi^n \right\|_{H^1} \le C_\psi + 1, \quad \left\| I_N \Phi^n \right\|_{H^1} \le C_\phi + 1, \quad \left\| I_N \dot{\Phi}^n \right\|_{H^1} \le \displaystyle \frac{1}{ \varepsilon^2}\left(C_\phi + 1\right). \label{thm1err3}
\end{align}
Thus, by taking the minimum of the two error estimates in \eqref{thm1err1}-\eqref{thm1err2} and then taking the maximum for $\varepsilon \in (0,1]$, we obtain an error estimate uniformly for $\varepsilon \in (0,1]$,
\begin{align}
   \left\| e^n_\psi \right\|_{H^1} + \left\| e^n_\phi \right\|_{H^1} + \varepsilon^2 \left\| \dot{e}^n_\phi \right\|_{H^1} \lesssim h^{m_0 -1} +  \max_{0 < \varepsilon \le 1}\min \left\{ \displaystyle \frac{\tau^2}{\varepsilon^2}, \varepsilon^2 \right\} \lesssim h^{m_0 -1} + \tau.\label{thm1err4}
\end{align}
\label{theorem1}
\end{Theorem}

\noindent In order to prove Theorem~\ref{theorem1}, we introduce the error energy functional \cite{bao2017uniformly}
\begin{align}
\mathcal{E}( e^n_\psi,e^n_\phi,\dot{e}^n_\phi) := \displaystyle \frac{1}{\varepsilon^2} \left( \left\| e^n_\psi \right\|_{H^1}^2 + \left\| e^n_\phi \right\|_{H^1}^2 \right) + \left\|\partial_x e^n_\phi \right\|_{H^1}^2 + \varepsilon^2 \left\| \dot{e}^n_\phi \right\|_{H^1}^2  , \quad n \ge 0. \label{enefun}
\end{align}
The proof will be split into four main steps to be presented in the following subsections as:
(i) estimates for the micro (or local) variables $v^n$ and $r^n$, (ii)
 error functions and the corresponding error equations, (iii) estimates for local truncation errors and nonlinear terms errors, and 
(iv) proof for the main result by the energy method.

\subsection{Estimates for the micro (or local) variables}
We first show the prior estimates for $v^n$ and $r^n$ in the multiscale decomposition \eqref{v1d}-\eqref{r1d} at each time step.
\begin{Lemma}\label{lemma:prior}
Under the assumption (A), there exists a constant $\tau_1 >0$ independent of $0 < \varepsilon \le 1$ and $h>0$, such that for $0 < \tau \le \tau_1$
\begin{align}
& \left\| v^n \right\|_{L^{\infty}([0,\tau]; H^{m_0 })} + \left\| \partial_s v^n \right\|_{L^{\infty}([0,\tau]; H^{m_0 - 2})} + \varepsilon^2 \left\| \partial_{ss} v^n \right\|_{L^{\infty}([0,\tau]; H^{m_0-2})} \lesssim 1, \label{lm1v} \\ 
& \left\| r^n \right\|_{L^{\infty}([0,\tau]; H^{m_0 -2 })} + \varepsilon^2 \left\| \partial_s r^n \right\|_{L^{\infty}([0,\tau]; H^{m_0 -2})} + \varepsilon^4 \left\| \partial_{ss} r^n \right\|_{L^{\infty}([0,\tau]; H^{m_0-2})} \lesssim \varepsilon^2. \label{lm1r}
\end{align}
\end{Lemma}
\begin{proof}
Recalling the equation of $v^n$ in \eqref{v1d} and applying the Duhamel's principle, we get
\begin{align}
    v^n(x,s) = a(s) v^n(x,0),\qquad x \in \overline{\Omega}, \qquad 0 \le s \le \tau,
\end{align}
where $a(s)$ is a pseudo-differential operartor
\begin{align*}
    a(s) := \frac{\lambda^+e^{is \lambda^-} - \lambda^- e^{is \lambda^+}}{\lambda^+ - \lambda^-}, \qquad \lambda^\pm:= - \frac{1}{\varepsilon^2} \left[ 1 \pm \sqrt{1 - \varepsilon^2 \Delta }\right],
\end{align*}
and for any $f \in H^k(\Omega)$
\begin{align*}
    \left\| a(s) f \right\|_{H^k} \le 2 \left\| f \right\|_{H^k},\qquad 0 \le s \le \tau.
\end{align*}
Then recalling the initial data in \eqref{v1d}, we have for $ 0 \le s \le \tau$
\begin{align*}
    \left\| v^n(\cdot,s) \right\|_{H^{m_0}} \le 2 \left\| v^n(\cdot,0)\right\|_{H^{m_0}} \le \left\| \phi(\cdot,t_n) \right\|_{H^{m_0}} + \varepsilon^2 \left\| \partial_t \phi(\cdot,t_n) \right\|_{H^{m_0}}, 
\end{align*}
which implies
\begin{align*}
    \left\| v^n \right\|_{L^{\infty}([0,\tau];H^{m_0})} \le \left\| \phi \right\|_{L^{\infty}([0,T];H^{m_0})} + \varepsilon^2 \left\| \partial_t \phi \right\|_{L^{\infty}([0,T];H^{m_0})} \lesssim 1, \quad n \ge 0.
\end{align*}
Similarly, by differentiating \eqref{v1d}, we can get the formula for $\partial_s v^n$, i.e.,
\begin{align}
    \partial_s v^n(x,s) = \varepsilon^2 b(s) \partial_{ss}v^n(x,0), \qquad b(s) := i \frac{e^{is \lambda^+} - e^{is \lambda^-}}{\varepsilon^2 (\lambda^- - \lambda^+ )}, \label{lm1.1}
\end{align}
with 
\begin{align}
    \partial_{ss} v^n(x,0) = \frac{1}{\varepsilon^2} \partial_{xx} v^n(x,0) \in H^{m_0 -2}.\label{lm1.2}
\end{align}
Thus, with
\begin{align}\label{lm1.3}
    \varepsilon^2 \partial_{ss} v^n(x,s) = -2i \partial_s v^n(x,s) + \partial_{xx}v^n(x,s), \quad x \in \Omega,\quad 0 \le s \le \tau,
\end{align}
we combine \eqref{lm1.1}-\eqref{lm1.3} and establish estimates \eqref{lm1v} in a similar manner with details omitted here for brevity.

To obtain the estimate \eqref{lm1r}, we perform the analysis in the Fourier space. Assuming
\begin{align*}
    r^n(x,s) = \sum_{l \in \mathbb{Z} }\widehat{(r^n)}_l(s) e^{i\mu_l(x-a)}, \qquad x \in \overline{\Omega}, \qquad 0 \le s \le \tau,
\end{align*}
and taking the Fourier transform on both sides of \eqref{r1d}, we have
\begin{align*}
    \widehat{(r^n)}_l(s) & = \int_0^s \frac{\sin(\omega_l(s - \theta))}{\varepsilon^2 \omega_l}\widehat{(\left| \psi^n\right|^2)}_l(\theta) d \theta \nonumber \\
    & = \int_0^s\frac{1}{\varepsilon^2 \omega_l^2} \widehat{(\left| \psi^n\right|^2)}_l(\theta) d ( \cos(\omega_l(s - \theta))\nonumber \\
    & = \frac{1}{\varepsilon^2 \omega_l^2} \widehat{(\left| \psi^n\right|^2)}_l(s) - \frac{\cos(\omega_l s)}{\varepsilon^2 \omega_l^2} \widehat{(\left| \psi^n\right|^2)}_l(0) - \int_0^s \frac{\cos(\omega_l(s - \theta))}{\varepsilon^2 \omega_l^2} \widehat{(\left| \psi^n\right|^2)}_l'(\theta) d\theta,
\end{align*}
which implies
\begin{align}
    \left| \widehat{(r^n)}_l(s) \right| \le \varepsilon^2  \left( \left| \widehat{(\left| \psi^n \right|^2)}_l(s) \right| +  \left| \widehat{(\left| \psi^n \right|^2)}_l(0) \right| + \int_0^s \left| \widehat{(\left| \psi^n \right|^2)}_l'(\theta) \right| d\theta \right). \label{lm1.4}
\end{align}
In order to estimate $\partial_s|\psi^n(s)|^2 = \partial_s|\psi(t_n + s)|^2$, multiplying both sides of \eqref{psi1d} by $\overline{\psi}$ and taking the imaginary part, we have
\begin{align}
    i \partial_t \rho(x,t) + \partial_x \left[ \overline{\psi(x,t)} \partial_x \psi(x,t) - \psi(x,t) \overline{\partial_x \psi(x,t)}\right]=0, \quad x \in \Omega, \quad t > 0,
\end{align}
where $\rho(x,t):= \left| \psi(x,t) \right|^2$. Under the assumption (A), we have
\begin{align}
    \left\| \partial_t^k \rho \right\|_{L^{\infty}([0,T];H^{m_0 -2k })} \lesssim 1, \qquad k = 0,1,2. \label{rho}
\end{align}
Then multiplying the square of \eqref{lm1.4} by $1 + \mu_l^2 + \ldots + \mu_l^{2m_0-4}$, and summing them up for $l \in \mathbb{Z}$, we get
\begin{align}
    \left\| r^n(\cdot,s) \right\|_{H^{m_0 -2}}^2 & \lesssim  \varepsilon^4 \left(  \left\| \rho^n (\cdot,s) \right\|_{H^{m_0 -2 }}^2 + \left\| \rho^n (\cdot,0) \right\|_{H^{m_0-2}}^2  + \int_0^s \left\| \partial_\theta \rho^n (\cdot,\theta) \right\|_{H^{m_0 -2}}^2 d\theta \right) \nonumber \\
    & \lesssim \varepsilon^4 \left(  \left\| \rho \right\|_{L^{\infty}([0,T];H^{m_0-2})}^2 + s^2\left\| \partial_t \rho \right\|_{L^{\infty}([0,T];H^{m_0-2})}^2 \right), \label{lm1.5}
\end{align}
where $\rho^n(x,s) := \rho(x,t_n + s)$ for $0 \le s \le \tau$.

Combining \eqref{rho} and \eqref{lm1.5}, we immediately obtain 
\begin{equation}
    \left\| r^n \right\|_{L^{\infty}([0,\tau];H^{m_0 -2})} \lesssim \varepsilon^2. \label{est_rn}
\end{equation}
Similarly, by differentiating the Fourier coefficient $\widehat{(r^n)}_l (s)$, it arrives at
\begin{align*}
\widehat{(r^n)}'_l(s) =  \int_0^s \frac{\cos(\omega_l(s - \theta))}{\varepsilon^2 }\widehat{(\left| \psi^n\right|^2)}_l(\theta) d \theta.
\end{align*}
Then combining integration by parts and \eqref{rho}, we can obtain the estimate
\begin{align}
     \left\| \partial_s r^n \right\|_{L^{\infty}([0,\tau];H^{m_0-2})} \lesssim 1, \qquad n \ge 0. \label{est_drn}
\end{align}
Recalling the equation of $r^n$ and the corresponding Duhamel's formula, i.e.,
\begin{align*}
& \varepsilon^2 \partial_{ss} r^n(x,s) = \Delta r^n(x,s) - \frac{1}{\varepsilon^2} r^n(x,s) +  |\psi^n(x,s)|^2, \\
& \widehat{(r^n)}_l (s) = \int_0^s \frac{\sin(\omega_l (s- \theta))}{\varepsilon^2 \omega_l} \widehat{(|\psi^n|^2)}_l(\theta) d\theta, 
\end{align*}
we have
\begin{align*}
   \left\| r^n \right\|_{L^{\infty}([0,\tau];H^{m_0})} \lesssim  \left\| \rho \right\|_{L^{\infty}([0,T];H^{m_0})} \lesssim 1, 
\end{align*}
which, together with \eqref{est_rn} and \eqref{est_drn}, immediately implies 
\begin{align}
&\varepsilon^2 \left\| \partial_{ss} r^n \right\|_{L^{\infty}([0,\tau];H^{m_0 -2})} \nonumber \\
&\lesssim \left\| r^n \right\|_{L^{\infty}([0,\tau];H^{m_0 })} + \frac{1}{\varepsilon^2} \left\| r^n \right\|_{L^{\infty}([0,\tau];H^{m_0 -2})} + \left\| \rho \right\|_{L^{\infty}([0,T];H^{m_0})} \lesssim 1. \label{est_ddrn}
\end{align}
Combining the estimates \eqref{est_rn}-\eqref{est_ddrn}, we arrive at the conclusion \eqref{lm1r}.
\end{proof}

\subsection{Error functions and the corresponding error equations}
In order to prove Theorem~\ref{theorem1}, we define another set of error functions
\begin{align*}
& e_{\psi,N}^n (x) := (P_N \psi)(x,t_n) - ( I_N \Psi^n ) (x), \quad x \in \overline{\Omega}, \quad n \ge 0, \\
& e_{\phi,N}^n (x) := (P_N \phi)(x,t_n) - ( I_N \Phi^n ) (x), \quad \dot{e}_{\phi,N}^n := (P_N \partial_t \phi)(x,t_n) - ( I_N \dot{\Phi}^n ) (x). 
\end{align*} 
By the regularity of the solutions in assumption (A), we have
\begin{equation}\label{Tri_err}
\begin{aligned}
&\left\| e^n_\psi \right\|_{H^1} \le \left\| \psi(\cdot,t_n) - (P_N \psi)(\cdot,t_n) \right\|_{H^1} + \left\| e^n_{\psi,N} \right\|_{H^1} \lesssim h^{m_0 - 1} + \left\| e_{\psi,N}^n \right\|_{H^1} , \\
& \left\| e^n_\phi \right\|_{H^1} \le \left\| \phi(\cdot,t_n) - (P_N \phi)(\cdot,t_n) \right\|_{H^1} + \left\| e^n_{\phi,N} \right\|_{H^1} \lesssim h^{m_0 -1} + \left\| e_{\phi,N}^n \right\|_{H^1} , \\
& \left\| \dot{e}^n_\phi \right\|_{H^1} \le \left\| \partial_t \phi(\cdot,t_n) - (P_N \partial_t \phi)(\cdot,t_n) \right\|_{H^1} + \left\| \dot{e}_{\phi,N}^n \right\|_{H^1} \lesssim \displaystyle \frac{h^{m_0 - 1}}{\varepsilon^2} + \left\| \dot{e}^n_{\phi,N} \right\|_{H^1}.
\end{aligned}
\end{equation}
Thus, we only need to prove estimates \eqref{thm1err1} and \eqref{thm1err2} with $e_\psi^n$, $e^n_\phi$ and $\dot{e}^n_\phi$
replaced by $e^n_{\psi,N}$, $e^n_{\phi,N}$ and $\dot{e}_{\phi,N}^n$, respectively. 

Before we introduce the error functions, we first give out the formula of the exact solutions $\psi$ and $\phi$, and reformulate the MTI-FP method \eqref{MTIFP}-\eqref{init00}.

\begin{Lemma}[Formula of exact solutions]\label{lemma:exactsol}
 Denote the Fourier expansion of the exact solutions $\psi(x,t)$ and $\phi(x,t)$ of the KGS system \eqref{psi1d}-\eqref{phi1d} as 
\begin{align}
\psi(x,t) = \sum_{l \in \mathbb{Z}} \widehat{\psi}_l(t) e^{i \mu_l (x-a)},\quad \phi(x,t) = \sum_{l \in \mathbb{Z}} \widehat{\phi}_l(t) e^{i \mu_l (x-a)}, \quad x \in \overline{\Omega}, \quad t \ge0,
\end{align}
then we have
\begin{equation}\label{FSexact}
\left\{\begin{aligned}
\widehat{\psi}_l(t_{n+1}) = &\ e^{-i\mu_l^2 \tau} \widehat{\psi}_l(t_n) + i e^{-i\mu_l^2 \tau} \int_0^\tau e^{i\left( \mu_l^2 + \frac{1}{\varepsilon^2} \right) \theta} \widehat{(v^n \psi^n)}_l(\theta) d \theta  \\
&\ + i e^{-i\mu_l^2 \tau} \int_0^\tau e^{i\left( \mu_l^2 - \frac{1}{\varepsilon^2} \right) \theta} \widehat{(\overline{v^n} \psi^n)}_l(\theta) d \theta + i \int_0^\tau e^{i\mu_l^2(\theta - \tau)} \widehat{(r^n \psi^n)}_l (\theta) d\theta, \\
\widehat{\phi}_l(t_{n+1}) = &\ \cos(\omega_l\tau) \widehat{\phi}_l(t_n) + \frac{\sin(\omega_l \tau)}{\omega_l} \widehat{\phi}_l' (t_n) + \int_0^{\tau} \frac{\sin(\omega_l (\tau - \theta))}{\varepsilon^2 \omega_l} \widehat{(  \rho^n )}_l(\theta) d\theta,  \\
\widehat{\phi}_l'(t_{n+1}) = &\  \cos(\omega_l \tau) \widehat{\phi}_l'(t_n) -\omega_l \sin(\omega_l \tau) \widehat{\phi}_l(t_n) + \int_0^{\tau} \frac{\cos(\omega_l (\tau - \theta))}{\varepsilon^2} \widehat{(\rho^n)}_l(\theta) d\theta,
\end{aligned}\right.
\end{equation}
where $\psi^n(\theta)= \psi(t_n + \theta)$ and $\rho^n(\theta)= \left| \psi (t_n + \theta)\right|^2$.
\end{Lemma}
\begin{proof}
    Recalling the problem \eqref{psi1d}-\eqref{phi1d} on the interval $[t_n,t_{n+1}]$ and taking the Fourier transform, we have for $l \in \mathbb{Z}$ and $0 \le s \le \tau$
\begin{equation*}
\left\{\begin{aligned}
    & i \widehat{\psi}_l'(t_n + s) - \mu_l^2 \widehat{\psi}_l (t_n + s) + \widehat{(\phi \psi)}_l(t_n + s) = 0, \\
    & \varepsilon^2 \widehat{\phi}_l''(t_n + s) + \mu_l^2 \widehat{\phi}_l ( t_n + s) + \frac{1}{\varepsilon^2} \widehat{\phi}_l(t_n + s) = \widehat{( \left| \psi \right|^2)}_l(t_n + s).
\end{aligned}\right. 
\end{equation*}

Noticing the decomposition \eqref{ansatz1d} and taking $s = \tau$, we could prove \eqref{FSexact} by the Duhamel's principle and the details are omitted here for brevity. 
\end{proof}

\begin{Lemma}[Reformulation of MTI-FP]\label{lemma_reformulation} For $n \ge 0$, expanding numerical solutions $(I_N\Phi^n)(x)$ and $(I_N \dot{\Phi}^n)(x)$ into Fourier series as
\begin{align}\label{Reformulation}
(I_N \Phi^n)(x) = \sum_{l = -N/2}^{N/2-1} \widetilde{(\Phi^n)}_l e^{i\mu_l(x-a)},\quad (I_N \dot{\Phi}^n)(x) = \sum_{l = -N/2}^{N/2-1} \widetilde{(\dot{\Phi}^n)}_l e^{i\mu_l(x-a)},
\end{align}
then we have for $l = -\frac{N}{2},\ldots,\frac{N}{2}-1$,
\begin{equation}
\label{ReCoeff}
\left\{\begin{aligned}
\widetilde{(\Phi^{n+1})}_l = &\ \cos(\omega_l\tau) \widetilde{(\Phi^n)}_l + \frac{\sin(\omega_l \tau)}{\omega_l} \widetilde{(\dot{\Phi}^n)}_l \\
&\ + p_l(\tau) \widetilde{(\left| \Psi^n \right|^2)}_l + 2 q_l(\tau) \widetilde{( \mathrm{Re}\{\overline{\Psi^n} \dot{\Psi}^n\})}_l, \\
\widetilde{(\dot{\Phi}^{n+1})}_l = &\ \cos(\omega_l\tau) \widetilde{(\dot{\Phi}^n)}_l -\omega_l \sin(\omega_l \tau) \widetilde{(\Phi^n)}_l  \\
&\ + p_l'(\tau) \widetilde{(\left| \Psi^n \right|^2)}_l + 2 q_l'(\tau) \widetilde{( \mathrm{Re}\{\overline{\Psi^n} \dot{\Psi}^n\})}_l. 
\end{aligned}\right.
\end{equation}
\end{Lemma}
\begin{proof}
Recalling the MTI-FP method \eqref{MTIFP}-\eqref{MTIFPvr}, we have for $l = -\frac{N}{2},\ldots,\frac{N}{2}-1$
\begin{align}
\widetilde{\left(\Phi^{n+1}\right)}_l = e^{i\tau / \varepsilon^2}\widetilde{\left(\mathbf{v}^{n,1}\right)}_l + e^{-i\tau / \varepsilon^2}\widetilde{\left(\overline{\mathbf{v}^{n,1}}\right)}_l + \widetilde{\left(\mathbf{r}^{n,1}\right)}_l. \label{Lm2.1}
\end{align}
Plugging the formula \eqref{MTIFP_Covrp} for $\widetilde{(\mathbf{v}^{n,1})}_l$ and $\widetilde{(\mathbf{r}^{n,1})}_l$ and the initial data \eqref{MTIFP_Ini} into \eqref{Lm2.1}, we obtain for $l = -\frac{N}{2},\ldots,\frac{N}{2}-1$
\begin{align}
    \widetilde{(\Phi^{n+1})}_l = &\  \mathrm{Re} \left( e^{i\tau / \varepsilon^2} a_l(\tau)\right) \widetilde{(\Phi^n)}_l + \varepsilon^2 \mathrm{Im} \left( e^{i\tau / \varepsilon^2} a_l(\tau) \right) \widetilde{(\dot{\Phi}^n)}_l \nonumber \\
    &\ + p_l(\tau) \widetilde{\left(\left| \Psi^n \right|^2\right)}_l + 2 q_l(\tau) \widetilde{\left( \mathrm{Re}\left(\overline{\Psi^n} \dot{\Psi}^n\right)\right)}_l, \label{lm2.2}
\end{align}
where $\mathrm{Im}(f)$ represents the imaginary part of the function $f$.

Noticing the coefficient \eqref{Coeffa}, we have 
\begin{equation*}
    \mathrm{Re} \left( e^{i\tau / \varepsilon^2} a_l(\tau)\right) = \cos(\omega_l\tau), \qquad \mathrm{Im} \left( e^{i\tau / \varepsilon^2} a_l(\tau)\right) = \frac{\sin(\omega_l \tau)}{\varepsilon^2 \omega_l},
\end{equation*}
which combined with \eqref{lm2.2} completes the proof of the first equation in \eqref{ReCoeff}. The proof for the second equation follows similarly, and we omit the details here for brevity.
\end{proof}

Then, we introduce the local truncation error functions as
\begin{equation}\label{Local}
\begin{aligned}
    & \xi_\psi^n(x) = \sum_{l = -N/2}^{N/2-1} \widehat{(\xi_\psi^n)}_l e^{i\mu_l(x-a)}, \qquad x \in \overline{\Omega},\qquad n \ge 0, \\
    &\xi_\phi^n(x) = \sum_{l = -N/2}^{N/2-1} \widehat{(\xi_\phi^n)}_l e^{i\mu_l(x-a)}, \quad \dot{\xi}_\phi^n(x) = \sum_{l = -N/2}^{N/2-1} \widehat{(\dot{\xi}_\phi^n)}_l e^{i\mu_l(x-a)},
\end{aligned}
\end{equation}
where
\begin{subequations}\label{LocalCo}
\begin{align}
\widehat{(\xi^n_\psi)}_l := &\ \widehat{\psi}_l(t_{n+1}) -\left[  e^{-i\mu_l^2 \tau} \widehat{\psi}_l(t_n) +  c_l^+(\tau) \widehat{(v^n \psi^n)}_l(0) +  d_l^+(\tau) \widehat{(v^n \delta^n)}_l(0) \right. \nonumber \\
&\left. +  \ c_l^-(\tau) \widehat{(\overline{v^n} \psi^n)}_l(0) + d_l^-(\tau) \widehat{(\overline{v^n} \delta^n)}_l(0) \right] , \label{Localpsi} \\
\widehat{(\xi^n_\phi)}_l :=&\  \widehat{\phi}_l(t_{n+1}) - \left[ \cos(\omega_l \tau) \widehat{\phi}_l(t_n) +  \frac{\sin(\omega_l \tau)}{\omega_l} \widehat{\phi}_l'(t_n) + p_l(\tau) \widehat{(\left| \psi^n \right|^2)}_l(0) \right. \nonumber \\
& \left. \ + 2 q_l(\tau) \widehat{\left( \mathrm{Re}\left(\overline{\psi^n}\delta^n\right)\right)}_l(0) \right], \label{Localphi}\\
\widehat{(\dot{\xi}^n_\phi)}_l := &\ \widehat{\phi}_l'(t_{n+1})- \left[ \cos(\omega_l \tau) \widehat{\phi}_l'(t_n) -\omega_l \sin(\omega_l\tau) \widehat{\phi}_l(t_n) + p_l'(\tau) \widehat{(\left| \psi^n \right|^2)}_l(0) \right.\nonumber \\
& \left. \ + 2 q_l'(\tau) \widehat{\left( \mathrm{Re}\left(\overline{\psi^n}\delta^n\right)\right)}_l(0) \right], \label{Localdphi}
\end{align}
\end{subequations}
with the auxiliary function defined as
\begin{equation}\label{delta}
\begin{aligned} 
&\delta^n(x) = \sum_{l \in \mathbb{Z}} \widehat{(\delta^n)}_l e^{i\mu_l(x-a)},\qquad x \in \overline{\Omega},\\
&\widehat{(\delta^n)}_l := -i\frac{\sin(\mu_l^2\tau)}{\tau} \widehat{(\psi^n)}_l(0) + i \widehat{(\phi^n\psi^n)}_l(0).
\end{aligned}
\end{equation}
In addition, define the errors from the nonlinear terms as 
\begin{equation}\label{NL}
\begin{aligned}
    & \eta_\psi^n(x) = \sum_{l = -N/2}^{N/2-1} \widehat{(\eta_\psi^n)}_l e^{i\mu_l(x-a)}, \qquad x \in \overline{\Omega},\qquad n \ge 0, \\
    &\eta_\phi^n(x) = \sum_{l = -N/2}^{N/2-1} \widehat{(\eta_\phi^n)}_l e^{i\mu_l(x-a)}, \quad \dot{\eta}_\phi^n(x) = \sum_{l = -N/2}^{N/2-1} \widehat{(\dot{\eta}_\phi^n)}_l e^{i\mu_l(x-a)},
\end{aligned}
\end{equation}
where
\begin{subequations}\label{FNL}
\begin{align}
\widehat{(\eta_\psi^n)}_l := &\ c_l^+(\tau) \left[\widehat{\left(v^n \psi^n\right)}_l(0) -\widetilde{\left(\mathbf{v}^{n,0} \Psi^n\right)}_l \right] + d_l^+(\tau) \left[ \widehat{(v^n \delta^n)}_l(0) -\widetilde{\left(\mathbf{v}^{n,0} \dot{\Psi}^n\right)}_l \right] \nonumber \\
&\ + c_l^-(\tau) \left[ \widehat{\left(\overline{v^n} \psi^n\right)}_l(0)- \widetilde{\left(\overline{\mathbf{v}^{n,0}} \Psi^n\right)}_l \right] + d_l^-(\tau) \left[ \widehat{\left(\overline{v^n} \delta^n\right)}_l(0) -  \widetilde{\left(\overline{\mathbf{v}^{n,0}} \dot{\Psi}^n\right)}_l  \right],   \label{FNLpsi}\\
\widehat{\left(\eta_\phi^n\right)}_l := &\ p_l(\tau) \left[   \widehat{\left(\left| \psi^n \right|^2\right)}_l(0) - \widetilde{\left(\left| \Psi^n \right|^2\right)}_l \right] \nonumber \\
& \ + 2 q_l(\tau) \left[  \widehat{\left( \mathrm{Re}\left(\overline{\psi^n}\delta^n\right)\right)}_l(0) - \widetilde{\left( \mathrm{Re}\left(\overline{\Psi^n} \dot{\Psi}^n\right)\right)}_l \right], \label{FNLphi}\\
\widehat{\left(\dot{\eta}_\phi^n\right)}_l := &\ p_l'(\tau) \left[ \widehat{\left(\left| \psi^n \right|^2\right)}_l(0) - \widetilde{\left(\left| \Psi^n \right|^2\right)}_l \right] \nonumber \\
& \ + 2 q_l'(\tau) \left[ \widehat{\left( \mathrm{Re}\left(\overline{\psi^n}\delta^n\right)\right)}_l(0) - \widetilde{\left( \mathrm{Re}\left(\overline{\Psi^n} \dot{\Psi}^n\right)\right)}_l  \right]. \label{FNLdphi}
\end{align}
\end{subequations}
Subtracting the exact flow \eqref{FSexact} from the numerical flow \eqref{MTIFP.CPsi} and \eqref{ReCoeff}, we obtain the following error equations in the Fourier space
\begin{equation}\label{Erreq}
\begin{aligned}
  & \widehat{(e^{n+1}_{\psi,N})}_l = e^{-i\mu_l^2 \tau} \widehat{(e^n_{\psi,N})}_l  + \widehat{(\xi^n_\psi)}_l + \widehat{(\eta^n_\psi)}_l , \qquad l = -\frac{N}{2},\ldots,\frac{N}{2}-1,  \\
  & \widehat{(e^{n+1}_{\phi,N})}_l = \cos(\omega_l\tau) \widehat{(e^n_{\phi,N})}_l + \frac{\sin(\omega_l\tau)}{\omega_l} \widehat{(\dot{e}^n_{\phi,N})}_l + \widehat{(\xi^n_\phi)}_l + \widehat{(\eta^n_\phi)}_l  ,\\
  & \widehat{(\dot{e}^{n+1}_{\phi,N})}_l = -\omega_l \sin(\omega_l\tau) \widehat{(e^n_{\phi,N})}_l + \cos(\omega_l \tau) \widehat{(\dot{e}^n_{\phi,N})}_l + \widehat{(\dot{\xi}^n_\phi)}_l + \widehat{(\dot{\eta}^n_\phi)}_l.
\end{aligned}
\end{equation}

\subsection{Energy estimates for error functions}
For the local truncation error functions \eqref{Local}-\eqref{LocalCo}, we have the following estimates.
\begin{Lemma}\label{lemma:local}
Under the assumption (A), when $0 < \tau \le \tau_1$, we have two independent estimates for $0 < \varepsilon \le 1$
    \begin{align}
        \mathcal{E}(\xi^n_\psi,\xi^n_\phi,\dot{\xi}^n_\psi) \lesssim \frac{\tau^6}{\varepsilon^6} ,\qquad \mathcal{E}(\xi^n_\psi,\xi^n_\phi,\dot{\xi}^n_\psi) \lesssim \tau^2 \varepsilon^2,\qquad n = 0,1,\ldots,\frac{T}{\tau}-1. \label{est_lol}
    \end{align}
\end{Lemma}
\begin{proof}
We first derive two independent estimates for $\xi^n_\phi$. On the one hand, plugging the exact solution \eqref{FSexact} into \eqref{Localphi} and applying Taylor's expansion, we get
\begin{align*}
    \widehat{(\xi^n_\phi)}_l = \int_0^{\tau} \frac{\sin(\omega_l(\tau - \theta))}{\varepsilon^2 \omega_l} \theta^2 \left[ \int_0^1(1-s) \widehat{(\left| \psi^n \right|^2)}_l''(\theta s) ds \right] d\theta + 2 q_l(\tau) \widehat{(T_0)}_l,
\end{align*}
where 
\begin{equation*}
    T_0 (x):= \mathrm{Re}\left( \overline{\psi^n (x,0)} \Bigl( \partial_s \psi^n(x,0) - \delta^n(x) \Bigr) \right),
\end{equation*}
with $\delta^n$ defined in  \eqref{delta}.

Noticing $\frac{1}{\varepsilon^2 \omega_l} = \frac{1}{\sqrt{1 + \varepsilon^2 \mu_l^2}} \le 1$ and $\left| q_l(\tau) \right| \le \tau^2$ for $l = -\frac{N}{2},\ldots,\frac{N}{2}-1$, we have
\begin{align}
   \left| \widehat{\left(\xi^n_\phi\right)}_l \right| \le \tau^2 \int_0^{\tau} \int_0^1 \left| \widehat{\left(\left| \psi^n \right|^2\right)}_l''(\theta s) \right| ds  d\theta + 2 \tau^2 \left| \widehat{(T_0)}_l \right|. \label{lmphi1.1}
\end{align}
Thus, by using the Cauchy-Schwarz inequality, i.e.,
\begin{align*}
    \left( \int^\tau_0\int_0^1 \left| \widehat{\left(\left| \psi^n \right|^2\right)}_l''(\theta s) \right| ds d\theta \right)^2 & \le \int_0^\tau 1 d\theta \cdot \int_0^\tau  \left( \int_0^1 \left| \widehat{\left(\left| \psi^n \right|^2\right)}_l''(\theta s) \right| ds \right)^2 d\theta \\
    & \le \tau \int_0^\tau \int_0^1 \left| \widehat{\left(\left| \psi^n \right|^2\right)}_l''(\theta s) \right|^2 ds d\theta,
\end{align*}
and noticing $\rho^n(x,s) := \left| \psi^n (x,s) \right|^2$, we can obtain
\begin{align}
\left\| \xi^n_\phi \right\|_{H^1}^2 & \lesssim \tau^5 \int_0^\tau \int_0^1 \left\| \partial_{ss} \rho^n (\theta s ) \right\|_{H^1}^2 ds d\theta + \tau^4 \left\| T_0\right\|_{H^1}^2 \nonumber \\
& \lesssim \tau^6 \left\| \partial_{ss} \rho^n \right\|_{L^{\infty}([0,\tau];H^1)}^2 + \tau^4 \left\|  \overline{\psi^n (\cdot,0)} \Bigl( \partial_s \psi^n(\cdot,0) - \delta^n(\cdot) \Bigr) \right\|_{H^1}^2. \label{lmphi1.2}
\end{align}
Under the assumption (A), applying Sobolev embedding theorem \cite{evans2022partial}, we have
\begin{align}
\left\|  \overline{\psi^n (\cdot,0)} \Bigl( \partial_s \psi^n(\cdot,0) - \delta^n(\cdot) \Bigr) \right\|_{H^1}^2  & 
\lesssim \left\| \psi\right\|_{L^{\infty}([0,T];H^3)}^2 \left\|   \partial_s \psi^n(\cdot,0) - \delta^n(\cdot) \right\|_{H^1}^2  \nonumber \\
& \lesssim \left\|   \partial_s \psi^n(\cdot,0) - \delta^n(\cdot) \right\|_{H^1}^2. \label{lmphi1.3}
\end{align}
Noting
\begin{align*}
    \left| \sin(x) - x \right| \le 2 \left| x \right|^\alpha, \quad \forall x \in \mathbb{R},\quad 1 \le \alpha \le 3,
\end{align*}
with the definition of $\delta^n$ in \eqref{delta} and the Parseval's identity, we have
\begin{align}
    \left\|   \partial_s \psi^n(\cdot,0) - \delta^n(\cdot) \right\|_{H^1}^2 & = \sum_{l \in \mathbb{Z}} \left( 1 + \mu_l^2 \right) \left| \frac{\sin(\mu_l^2 \tau)}{\tau} - \mu_l^2 \right|^2 \left| \widehat{(\psi^n)}_l(0) \right|^2 \nonumber \\
    & \lesssim \tau^2 \sum_{l \in \mathbb{Z}} \mu_l^{10}  \left| \widehat{(\psi^n)}_l(0) \right|^2 \lesssim \tau^2 \left\| \psi^n \right\|_{L^{\infty}([0,\tau];H^5)}^2. \label{lmphi1.4}
\end{align}
Plugging \eqref{lmphi1.3} and \eqref{lmphi1.4} into \eqref{lmphi1.2} and noticing the assumption (A) and \eqref{rho}, we arrive at
\begin{align}
    \left\| \xi^n_\phi\right\|_{H^1}^2 \lesssim \tau^6, \qquad n \ge 0. \label{lmphi1.5}
\end{align}
Similarly, by noticing $\left| \frac{\mu_l}{\varepsilon^2 \omega_l}\right| \le 1/ \varepsilon
$ and $ \tau p_l'(\tau) + q_l'(\tau) \lesssim \tau^2 / \varepsilon^2$, we have
\begin{align}
    \left\| \partial_x \xi^n_\phi\right\|_{H^1}^2 \lesssim \frac{\tau^6}{\varepsilon^2}, \qquad \left\| \dot{\xi}^n_\phi\right\|_{H^1}^2 \lesssim \frac{\tau^6}{\varepsilon^4}, \qquad n \ge 0.  \label{lmphi1.6}
\end{align}

On the other hand, by applying Taylor's expansion truncated at the first order term, we get
\begin{align}
    \widehat{(\xi^n_{\phi})}_l = \int_0^{\tau}\frac{\sin(\omega_l(\tau - \theta))}{\varepsilon^2 \omega_l} \theta \left[ \int_0^1 \widehat{\left(\left| \psi^n \right|^2\right)}_l'(\theta s ) ds \right] d\theta - 2 q_l(\tau) \widehat{\left(\mathrm{Re}\left( \overline{\psi^n} \delta^n \right)\right)}_l(0). \label{lmphi2.1}
\end{align}
Define the following functions
\begin{equation}\label{lmphi2.2}
\begin{aligned}
    & Q_l(\theta) = \frac{\theta}{\varepsilon^2 \omega_l^2} \cos(\omega_l(\tau - \theta)) + \frac{\sin(\omega_l(\tau - \theta)) - \sin(\omega_l\tau)}{\varepsilon^2 \omega_l^3} ,\\
    & Q_l'(\theta) = \theta \frac{\sin(\omega_l(\tau - \theta))}{\varepsilon 
    ^2\omega_l} , \qquad l = -\frac{N}{2},\ldots,\frac{N}{2}-1.
\end{aligned}
\end{equation}
Plugging \eqref{lmphi2.2} into \eqref{lmphi2.1}, then we get
\begin{align*}
    \widehat{(\xi^n_\phi)}_l  = & \  \int_0^\tau Q_l'(\theta) \left[ \int_0^1 \widehat{\left(\left| \psi^n \right|^2\right)}_l'(\theta s) ds  \right] d\theta - 2 q_l(\tau) \widehat{\left(\mathrm{Re}\left( \overline{\psi^n} \delta^n \right)\right)}_l(0)  \\
    = &\ Q_l(\tau) \int_0^1 \widehat{\left(\left| \psi^n \right|^2\right)}_l'(\tau s ) ds - \int_0^\tau Q_l(\theta) \left[ \int_0^1 s \widehat{\left(\left| \psi^n \right|^2\right)}_l''(\theta s ) ds\right] d\theta \\
    & \ - 2 q_l(\tau) \widehat{\left(\mathrm{Re}\left( \overline{\psi^n} \delta^n \right)\right)}_l(0) .
\end{align*}
Noticing $\left| q_l(\tau)\right| + \left| Q_l(\theta) \right| \lesssim \tau \varepsilon^2$ for $0 \le \theta \le \tau$ and $l = - \frac{N}{2},\ldots,\frac{N}{2}-1$, we obtain
\begin{align*}
\left| \widehat{(\xi^n_\phi)}_l \right| \lesssim & \tau \varepsilon^2 \left( \int_0^1 \left| \widehat{\left(\left| \psi^n \right|^2\right)}_l'(\tau s ) \right| ds + \int_0^\tau \int_0^1 \left| \widehat{\left(\left| \psi^n \right|^2\right)}_l''(\theta s ) \right| ds d\theta \right. \\
& \left. + \left| \widehat{\left(\mathrm{Re}\left( \overline{\psi^n} \delta^n \right)\right)}_l(0)  \right| \right),
\end{align*}
which implies
\begin{align*}
    \left\| \xi^n_\phi \right\|_{H^1}^2 \lesssim \tau^2 \varepsilon^4 \left( \left\| \partial_s \rho^n \right\|_{L^{\infty}([0,\tau];H^1)}^2 + \tau^2 \left\| \partial_{ss} \rho^n \right\|_{L^{\infty}([0,\tau];H^1)}^2 + \left\| \overline{\psi^n}\delta^n \right\|_{L^{\infty}([0,\tau];H^1)}^2 \right).
\end{align*}
Noting the estimates \eqref{rho}, the definition of $\delta^n$ \eqref{delta} and the assumption (A), we apply Sobolev embedding theorem \cite{evans2022partial} and obtain from the above equation that
\begin{align}
    \left\| \xi^n_\phi \right\|_{H^1}^2 \lesssim \tau^2\varepsilon^4. \label{lmphi2.3}
\end{align}
Similarly, with details omitted here, we also have
\begin{align} \label{lmphi2.4}
    \left\| \partial_x \xi^n_\phi \right\|_{H^1}^2 \lesssim \tau^2\varepsilon^2,\qquad \left\| \dot{\xi}^n_\phi \right\|_{H^1}^2 \lesssim \tau^2.
\end{align}

Next, we are going to carry out the error estimates for $\widehat{(\xi^n_\psi)}_l$. On the one hand, plugging the exact solution \eqref{FSexact} into \eqref{Localpsi} and using Taylor's expansion, we have
\begin{align}
\widehat{(\xi^n_\psi)}_l = &\  i e^{-i\mu_l^2 \tau} \int_0^\tau e^{i(\mu_l^2 + \frac{1}{\varepsilon^2}) \theta} \theta^2 \left[ \int_0^1 (1-s) \widehat{\left(v^n \psi^n\right)}_l''(\theta s ) ds \right] d\theta  \nonumber \\
&\ + i e^{-i\mu_l^2 \tau} \int_0^\tau e^{i(\mu_l^2 - \frac{1}{\varepsilon^2}) \theta} \theta^2 \left[ \int_0^1 (1-s) \widehat{\left(\overline{v^n} \psi^n \right)}_l''(\theta s ) ds \right] d\theta \nonumber \\
    &\ + d_l^+(\tau) \widehat{\left(v^n(\partial_s \psi^n - \delta^n)\right)}_l(0)  + d_l^-(\tau) \widehat{\left(\overline{v^n}(\partial_s \psi^n - \delta^n)\right)}_l(0) \nonumber \\
    &\ + i \int_0^\tau e^{i\mu_l^2 (\theta - \tau)} \theta^2 \left[ \int_0^1(1-s) \widehat{\left(r^n \psi^n\right)}_l''(\theta s ) ds \right] d\theta \label{lmpsi1.1},
\end{align}
where we apply the homogeneous initial data $\partial_s v^n(x,0) = r^n(x,0) = \partial_s r^n (x,0) \equiv 0$.
Thus, by noting $\left| d_l^\pm (\tau) \right| \le \tau^2$, we have from \eqref{lmpsi1.1} that
\begin{align}
\left\| \xi^n_\psi \right\|_{H^1}^2 \lesssim &\ \tau^6 \left(  \left\| \partial_{ss} \left( v^n \psi^n  \right)\right\|_{L^{\infty}([0,\tau];H^1)}^2 + \left\| \partial_{ss} \left( \overline{v^n} \psi^n  \right)\right\|_{L^{\infty}([0,\tau];H^1)}^2 \right) \nonumber \\
   &\  + \tau^4 \left( \left\| v^n\left( \partial_s \psi^n - \delta^n\right) \right\|_{L^{\infty}([0,\tau];H^1)}^2 + \left\| \overline{v^n} \left( \partial_s \psi^n - \delta^n\right) \right\|_{L^{\infty}([0,\tau];H^1)}^2  \right) \nonumber \\
   &\ +  \tau^6  \left\| \partial_{ss} \left( r^n \psi^n  \right)\right\|_{L^{\infty}([0,\tau];H^1)}^2 .\label{lmpsi1.2}
\end{align}
Similar to the estimate of \eqref{lmphi1.3}, using Sobolev embedding theorem and \eqref{lmphi1.4}, and applying prior estimates \eqref{lm1v} and \eqref{lm1r} to \eqref{lmpsi1.2}, we obtain
\begin{align}
    \left\| \xi^n_\psi \right\|_{H^1}^2 \lesssim \frac{\tau^6}{\varepsilon^4}.  \label{lmpsi1.3}
\end{align}

On the other hand, applying Taylor's expansion truncated at the first order term, we obtain
\begin{align}
    \widehat{(\xi^n_\psi)}_l = &\ i e^{-i\mu_l^2 \tau} \int_0^\tau e^{i\left(\mu_l^2 + \frac{1}{\varepsilon^2}\right) \theta} \theta \left[ \int_0^1  \widehat{\left(v^n \psi^n \right)}_l'(\theta s ) ds \right] d\theta - d_l^+(\tau) \widehat{\left(v^n \delta^n\right)}_l(0) \nonumber \\
    &\ + i e^{-i\mu_l^2 \tau} \int_0^\tau e^{i\left(\mu_l^2 - \frac{1}{\varepsilon^2}\right) \theta} \theta \left[ \int_0^1 \widehat{\left(\overline{v^n} \psi^n \right)}_l'(\theta s ) ds \right] d\theta  - d_l^-(\tau) \widehat{\left(\overline{v^n} \delta^n\right)}_l(0) \nonumber \\
    &\  + i \int_0^\tau e^{i\mu_l^2 (\theta - \tau)}  \widehat{\left(r^n \psi^n\right)}_l(\theta) d\theta. \label{xi_psi_esti}
\end{align}

Then applying integration by parts twice, i.e.,
\begin{align}
& \int_0^\tau e^{i\left(\mu_l^2 + \frac{1}{\varepsilon^2}\right) \theta} \theta \left[ \int_0^1  \widehat{(v^n \psi^n )}_l'(\theta s ) ds \right] d\theta \nonumber \\
& = -i \varepsilon^2\int_0^\tau e^{i \mu_l^2 \theta} \theta \left[ \int_0^1  \widehat{(v^n \psi^n )}_l'(\theta s ) ds \right] d(e^{i\theta/\varepsilon^2}) \nonumber \\
& = -i \varepsilon^2 \tau e^{i( \mu_l^2+\frac{1}{\varepsilon^2})\tau} \int_0^1  \widehat{\left(v^n \psi^n \right)}_l'(\tau s) ds   + i \varepsilon^2 \int_0^\tau  e^{i\left(\mu_l^2 + \frac{1}{\varepsilon^2}\right) \theta} \theta \left[ \int_0^1  s \widehat{\left(v^n \psi^n \right)}_l''(\theta s ) ds \right] d \theta \nonumber \\
& \quad + i \varepsilon^2 \int_0^\tau e^{i(\mu_l^2 + \frac{1}{\varepsilon^2}) \theta} (1 + i \mu_l^2 \theta) \left[ \int_0^1  \widehat{\left(v^n \psi^n \right)}_l'(\theta s) ds  \right] d\theta  \nonumber \\
& = -i \varepsilon^2 \tau e^{i\left(\mu_l^2+\frac{1}{\varepsilon^2}\right)\tau} \int_0^1  \widehat{\left(v^n \psi^n \right)}_l'(\tau s) ds   + i \varepsilon^2 \int_0^\tau  e^{i\left(\mu_l^2 + \frac{1}{\varepsilon^2}\right) \theta} \widehat{\left(v^n\psi^n\right)}'_l(\theta) d \theta \nonumber \\
& \quad -  \varepsilon^2 \mu_l^2 \int_0^\tau e^{i\left(\mu_l^2 + \frac{1}{\varepsilon^2}\right) \theta}  \theta \left[ \int_0^1  \widehat{\left(v^n \psi^n\right)}_l'(\theta s) ds  \right] d\theta, \label{xi_psi_1}
\end{align}
and noticing $\left| d_l^\pm (\tau) \right| \lesssim \tau \varepsilon^2 (1 + \mu_l^2)$ for $l = -\frac{N}{2},\ldots,\frac{N}{2}-1$, we combine \eqref{xi_psi_esti} and \eqref{xi_psi_1}, use Cauchy-Schwarz inequality and get
\begin{align}
\left\| \xi^n_\psi \right\|_{H^1}^2 \lesssim & \ \tau^2 \varepsilon^4 \left( \left\|v^n \delta^n \right\|_{L^{\infty}([0,\tau];H^3)}^2 + \left\| \partial_s(v^n \psi^n) \right\|_{L^{\infty}([0,\tau];H^3)}^2  + \left\|  \overline{v^n} \delta^n \right\|_{L^{\infty}([0,\tau];H^3)}^2  \right. \nonumber\\
&\ \left.  + \left\| \partial_s( \overline{v^n} \psi^n) \right\|_{L^{\infty}([0,\tau];H^3)}^2  \right)  + \tau^2 \left\|r^n \psi^n \right\|_{L^{\infty}([0,\tau];H^1)}^2  \nonumber\\
\lesssim &\ \tau^2 \varepsilon^4, \label{lmpsi2.1}
\end{align}
where we apply prior estimates \eqref{lm1v}-\eqref{lm1r} and Sobolev embedding theorem again.

Finally, recalling the energy functional \eqref{enefun} and combining \eqref{lmphi1.5}-\eqref{lmphi1.6} and \eqref{lmpsi1.3}, \eqref{lmphi2.3}-\eqref{lmphi2.4} and \eqref{lmpsi2.1}, respectively, we complete the proof.
\end{proof}

Then, we are going to derive the error estimates for the nonlinear terms \eqref{NL} and \eqref{FNL}.
\begin{Lemma}[Nonlinear terms error]\label{lemma:Non}
Under the assumption (A) and assuming \eqref{thm1err3} holds for n (which will be proved by induction later), we have for any $0 < \tau \le \tau_1$
\begin{align}
    \mathcal{E}(\eta^n_\psi,\eta^n_\phi,\dot{\eta}^n_\phi) \lesssim \tau^2 \mathcal{E}(e^n_{\psi,N},e^n_{\phi,N},\dot{e}^n_{\phi,N}) + \frac{\tau^2 h^{2 m_0-2}}{\varepsilon^2}. \label{lmNon_est}
\end{align}
\end{Lemma}

In order to prove Lemma~\ref{lemma:Non}, for any $\mathbf{v} \in Y_N$, we denote $v_{N+1}  = v_1$ and define the difference operator $\delta^+_x \mathbf{v} \in Y_N$ as
\begin{align*}
    \delta^+_x \mathbf{v}_j = \frac{v_{j+1} - v_j}{h}, \qquad j = 0,1,\ldots,N,
\end{align*}
with the corresponding discrete norm
\begin{align*}
    \left\| \mathbf{v} \right\|^2_{Y,1} := \left\| \mathbf{v} \right\|_{l^2}^2 + \left\| \delta^+_x \mathbf{v} \right\|_{l^2}^2.
\end{align*}
We have the following estimates \cite{bao2014uniform,bao2017uniformly}
\begin{align}\label{equiv_norm}
    \left\| I_N \mathbf{v} \right\|_{H^1} \lesssim \left\| \mathbf{v} \right\|_{Y,1} \lesssim \left\| I_N \mathbf{v} \right\|_{H^1}, \quad \forall \mathbf{v} \in Y_N.
\end{align}
\begin{proof} By the estimate $\tau \left| p_l(\tau)\right| + \left| q_l(\tau)\right| \lesssim \tau^2$ and applying the triangle inequality to \eqref{NL}-\eqref{FNL}, we obtain
\begin{align*}
\left\| \eta^n_\phi \right\|_{H^1} \lesssim & \tau \left\| P_N(\left|\psi^n\right|^2)(\cdot,0) - I_N (\left| \Psi^n \right|^2) \right\|_{H^1}   + \tau^2 \left\| P_N(\overline{\psi^n} \delta^n)(\cdot,0) - I_N (\overline{\Psi^n} \dot{\Psi}^n ) \right\|_{H^1}.
\end{align*}
Recalling the definition of $\delta^n$ in \eqref{delta} and the regularity assumption (A), we have
\begin{align}
    \left\| \delta^n (\cdot)\right\|_{H^{m_0}} \le \frac{1}{\tau}\left\| \psi (\cdot,t_n) \right\|_{H^{m_0}} +  \left\| \phi(\cdot,t_n) \psi(\cdot,t_n) \right\|_{H^{m_0}} \lesssim \frac{1}{\tau}. \label{est_delta}
\end{align}
Combining \eqref{est_delta} and assumption (A), we apply the triangle inequality and get
\begin{align}\label{lm5.1}
\left\| \eta^n_\phi \right\|_{H^1} \lesssim & \ \tau  \left\| I_N(\left|\psi^n\right|^2)(\cdot,0) - I_N (\left| \Psi^n \right|^2) \right\|_{H^1}  \nonumber \\
& \ + \tau^2 \left\| I_N(\overline{\psi^n} \delta^n)(\cdot,0) - I_N (\overline{\Psi^n} \dot{\Psi}^n ) \right\|_{H^1}  + \tau h^{m_0 -1}.
\end{align}
Applying the equivalence of norms \eqref{equiv_norm}, we get the estimates
\begin{align}\label{lm5.2}
& \left\| I_N(\left|\psi^n\right|^2)(\cdot,0) - I_N (\left| \Psi^n \right|^2) \right\|_{H^1} \nonumber \\
& \lesssim \left\| \left|\psi^n\right|^2(\cdot,0) - \left| \Psi^n \right|^2 \right\|_{Y,1}  \lesssim \left\| e_\psi^n \right\|_{Y,1} \lesssim \left\| e_\psi^n \right\|_{H^1},
\end{align}
and 
\begin{align}
& \left\| I_N(\overline{\psi^n} \delta^n)(\cdot,0) - I_N (\overline{\Psi^n} \dot{\Psi}^n ) \right\|_{H^1} 
 \lesssim  \left\| ( \overline{\psi^n} \delta^n ) (\cdot,0) -  \overline{\Psi^n} \dot{\Psi}^n  \right\|_{Y,1} \nonumber \\
& \lesssim   \left\| \psi^n(\cdot,0) -  \Psi^n  \right\|_{Y,1} + \left\|  \delta^n (\cdot,0) - \dot{\Psi}^n  \right\|_{Y,1} \nonumber \\
& \lesssim  \left\| e^n_\psi  \right\|_{Y,1} + \left\|  \partial_{xx}^S \psi^n (\cdot,0)- \partial_{xx}^S \Psi^n  \right\|_{Y,1} + \left\| ( \phi^n \psi^n) (\cdot,0) - \Phi^n\Psi^n \right\|_{Y,1} \nonumber \\
& \lesssim  \left\| e^n_\psi  \right\|_{H^1} + \left\| e^n_\phi  \right\|_{H^1} + \left\|  \partial_{xx}^S \psi^n (\cdot,0)- \partial_{xx}^S \Psi^n  \right\|_{H^1},\label{lm5.3}
\end{align}
where the operator $\partial_{xx}^S$ for a function $v(x)$ (or $\mathbf{v} \in Y_N$) is defined as
\begin{align*}
    \partial_{xx}^S v(x):= - \sum_{l\in\mathbb{Z}} \frac{\sin(\mu_l^2 \tau)}{\tau} \widehat{v}_l e^{i \mu_l(x-a)}, \quad  \partial_{xx}^S \mathbf{v}:= - \sum_{l = -N/2}^{N/2 - 1} \frac{\sin(\mu_l^2 \tau)}{\tau} \widetilde{\mathbf{v}}_l e^{i \mu_l(x-a)}.
\end{align*}
Similar to the estimate of \eqref{lm5.2} and by noticing $\left| \frac{\sin(\mu_l^2 \tau)}{\tau} \right| \le \frac{1}{\tau}$, we obtain from \eqref{lm5.3} that
\begin{align}
 \left\| I_N(\overline{\psi^n} \delta^n)(\cdot,0) - I_N (\overline{\Psi^n} \dot{\Psi}^n ) \right\|_{H^1}  \lesssim  \frac{1}{\tau} \left\| e^n_\psi  \right\|_{H^1}  + \left\| e^n_\phi  \right\|_{H^1}. \label{lm5.4}
\end{align}
Plugging \eqref{lm5.2} and \eqref{lm5.4} into \eqref{lm5.1}, and noticing \eqref{Tri_err}, we get
\begin{align}
    \left\| \eta^n_\phi \right\|_{H^1} \lesssim \tau \left( \left\| e_{\psi,N}^n \right\|_{H^1} + \left\| e_{\phi,N}^n \right\|_{H^1} + h^{m_0 - 1} \right). \label{lm5.5}
\end{align}
In addition, combining with $\tau \left| p_l'(\tau) \right| + \left| q_l'(\tau) \right| \lesssim \frac{\tau^2}{\varepsilon^2}$, we obtain similarly
\begin{align}
\left\| \partial_x \eta^n_\phi \right\|_{H^1}^2 + \varepsilon^2 \left\| \dot{\eta}^n_\phi 
\right\|_{H^1}^2 \lesssim \frac{\tau^2}{\varepsilon^2} \left( \left\| e^n_{\psi,N} \right\|_{H^1}^2 + \left\| e^n_{\phi,N}  \right\|_{H^1}^2 + h^{2 m_0 -2} \right). \label{lm5.6}
\end{align}

Similarly, recalling the initial data in \eqref{v1d} and \eqref{MTIIni1}, we apply the estimates \eqref{equiv_norm} and \eqref{est_delta}, and analyze terms in $\eta^n_\psi$ \eqref{FNLpsi}, i.e.,
\begin{align}
& \left\| P_N( v^n\psi^n )(\cdot,0) - I_N( \mathbf{v}^{n,0} \Psi^n) \right\|_{H^1} \nonumber \\
& \lesssim \left\| I_N( v^n\psi^n )(\cdot,0) - I_N( \mathbf{v}^{n,0} \Psi^n) \right\|_{H^1}  + h^{m_0 - 1} \nonumber \\
& \lesssim \left\| ( v^n\psi^n )(\cdot,0) -  \mathbf{v}^{n,0} \Psi^n \right\|_{Y,1}  + h^{m_0 - 1} \nonumber \\
& \lesssim  \left\| e_\phi^n \right\|_{H^1} +  \varepsilon^2 \left\| \dot{e}_\phi^n \right\|_{H^1} + \left\| e_\psi^n \right\|_{H^1} + h^{m_0 -1} , \label{lm5.7}
\end{align}
and 
\begin{align}
& \left\| P_N( v^n \delta^n )(\cdot,0) - I_N( \mathbf{v}^{n,0} \dot{\Psi}^n) \right\|_{H^1} \nonumber \\
& \lesssim \left\| I_N( v^n \delta^n )(\cdot,0) - I_N( \mathbf{v}^{n,0} \dot{\Psi}^n) \right\|_{H^1}  + \frac{h^{m_0 - 1}}{\tau}  \nonumber \\
& \lesssim \left\| ( v^n \delta^n )(\cdot,0) -  \mathbf{v}^{n,0} \dot{\Psi}^n \right\|_{Y,1}  + \frac{h^{m_0 - 1}}{\tau} \nonumber \\
& \lesssim \left\| e_\phi^n \right\|_{H^1} +  \varepsilon^2 \left\| \dot{e}_\phi^n \right\|_{H^1} + \frac{1}{\tau} \left\| e^n_\psi  \right\|_{H^1} + \frac{h^{m_0 - 1}}{\tau} . \label{lm5.8}
\end{align}
Combining \eqref{Tri_err} and \eqref{lm5.7}-\eqref{lm5.8} and noting  $\tau \left|c_l^\pm (\tau) \right| + \left| d_l^\pm(\tau) \right| \lesssim \tau^2 $, we get from \eqref{FNLpsi} that
\begin{align}
    \left\| \eta_\psi^n \right\|_{H^1} \lesssim  \tau \left(  \left\| e_{\phi,N}^n \right\|_{H^1} +  \varepsilon^2 \left\| \dot{e}_{\phi,N}^n \right\|_{H^1} + \left\| e_{\psi,N}^n \right\|_{H^1} + h^{m_0 -1} \right) . \label{lm5.9}
\end{align}
Recalling the energy functional \eqref{enefun}, we conclude the proof by \eqref{lm5.5}-\eqref{lm5.6} and \eqref{lm5.9}.
\end{proof}

\subsection{Proof of Theorem~\ref{theorem1} by the energy method}
\begin{proof}
For $ n= 0$, from the initial data \eqref{init00} and the assumption (A), we have
\begin{align*}
&\left\|e_\psi^0 \right\|_{H^1} + \left\|e_\phi^0 \right\|_{H^1} + \varepsilon^2 \left\|\dot{e}_\phi^0 \right\|_{H^1} \\
& = \left\| \psi_0-I_N \psi_0 \right\|_{H^1} + \left\| \phi_0 - I_N \phi_0 \right\|_{H^1} + \varepsilon^2 \left\|\phi_1 - I_N \phi_1 \right\|_{H^1} \lesssim h^{m_0 - 1}.
\end{align*}
In addition, using the triangle inequality, we know that there exists a constant $h_1 >0$ independent of $\varepsilon$ such that for $ 0 < h \le h_1$ and $\tau > 0$,
\begin{align*}
& \left\| I_N \Psi^0 \right\|_{H^1} \le \left\| e_\psi^0 \right\|_{H^1} + \left\| \psi_0 \right\|_{H^1} \le 1 + C_\psi,\\
& \left\| I_N \Phi^0 \right\|_{H^1} \le \left\| e_\phi^0 \right\|_{H^1} + \left\| \phi_0 \right\|_{H^1} \le 1 + C_\phi,\\
& \left\| I_N \dot{\Phi}^0 \right\|_{H^1} \le \left\| \dot{e}_\phi^0 \right\|_{H^1} + \frac{\left\| \phi_1 \right\|_{H^1}}{\varepsilon^2} \le \frac{1 + C_\phi}{\varepsilon^2}.
\end{align*}
Thus, the error bounds \eqref{thm1err1}-\eqref{thm1err3} hold for $n = 0$.

Now, we assume \eqref{thm1err1}-\eqref{thm1err3} are valid for $ n = 0,1,\ldots,m$, and prove the case for $n = m+1$. Using Cauchy's inequality, we obtain from the error equations \eqref{Erreq}
\begin{align*}
& \left| \widehat{(e^{n+1}_{\psi,N})}_l \right|^2 \le (1  +\tau) \left| \widehat{(e^{n}_{\psi,N})}_l \right|^2 + \frac{1 + \tau}{\tau} \left|\widehat{(\xi^n_\psi)}_l + \widehat{(\eta^n_\psi)}_l \right|^2,\\
& \left| \widehat{(e^{n+1}_{\phi,N})}_l \right|^2 \le (1  +\tau) \left| \cos(\omega_l\tau) \widehat{(e^{n}_{\phi,N})}_l + \frac{\sin(\omega_l \tau)}{\omega_l} \widehat{(\dot{e}^n_{\phi,N})}_l \right|^2 + \frac{1 + \tau}{\tau} \left|\widehat{(\xi^n_\phi)}_l + \widehat{(\eta^n_\phi)}_l \right|^2, \\
& \left| \widehat{(\dot{e}^{n+1}_{\phi,N})}_l \right|^2 \le (1  +\tau) \left| \omega_l \sin(\omega_l\tau) \widehat{(e^{n}_{\phi,N})}_l -\cos(\omega_l \tau) \widehat{(\dot{e}^n_{\phi,N})}_l \right|^2 + \frac{1 + \tau}{\tau} \left|\widehat{(\dot{\xi}^n_\phi)}_l + \widehat{(\dot{\eta}^n_\phi)}_l \right|^2.
\end{align*}
Multiplying the above first two equations by $(\mu_l^2 + \frac{1}{\varepsilon^2})(1 + \mu_l^2)$ and the third equation by $\varepsilon^2(1 + \mu_l^2)$, and then summing them up for $l = -\frac{N}{2},\ldots,\frac{N}{2}-1$, we get
\begin{align}
& \mathcal{E}(e^{n+1}_{\psi,N},e^{n+1}_{\phi,N},\dot{e}^{n+1}_{\phi,N}) \nonumber \\
& \le (1 + \tau) \mathcal{E}(e^{n}_{\psi,N},e^{n}_{\phi,N},\dot{e}^{n}_{\phi,N}) + \frac{1 + \tau}{\tau} \mathcal{E}(\xi^n_\psi + \eta^n_\psi,\xi^n_\phi + \eta^n_\phi,\dot{\xi}^n_\phi + \dot{\eta}^n_\phi) \nonumber \\
& \le  (1 + \tau) \mathcal{E}(e^{n}_{\psi,N},e^{n}_{\phi,N},\dot{e}^{n}_{\phi,N}) + \frac{1 + \tau}{\tau} \left[ \mathcal{E}(\xi^n_\psi,\xi^n_\phi,\dot{\xi}^n_\phi) + \mathcal{E}(\eta^n_\psi,\eta^n_\phi,\dot{\eta}^n_\phi)\right]. \label{thm.1}
\end{align}
Inserting \eqref{est_lol} in Lemma~\ref{lemma:local} and \eqref{lmNon_est} in Lemma~\ref{lemma:Non} into \eqref{thm.1}, respectively, we obtain two independent estimates
\begin{align*}
&\mathcal{E}(e^{n+1}_{\psi,N},e^{n+1}_{\phi,N},\dot{e}^{n+1}_{\phi,N}) - \mathcal{E}(e^{n}_{\psi,N},e^{n}_{\phi,N},\dot{e}^{n}_{\phi,N}) \lesssim \tau \mathcal{E}(e^{n}_{\psi,N},e^{n}_{\phi,N},\dot{e}^{n}_{\phi,N}) + \frac{\tau^5}{\varepsilon^6} + \frac{\tau h^{2m_0 -2}}{\varepsilon^2}, \\
& \mathcal{E}(e^{n+1}_{\psi,N},e^{n+1}_{\phi,N},\dot{e}^{n+1}_{\phi,N}) - \mathcal{E}(e^{n}_{\psi,N},e^{n}_{\phi,N},\dot{e}^{n}_{\phi,N}) \lesssim \tau \mathcal{E}(e^{n}_{\psi,N},e^{n}_{\phi,N},\dot{e}^{n}_{\phi,N}) + \tau \varepsilon^2 + \frac{ \tau h^{2m_0 -2}}{\varepsilon^2}.
\end{align*}
Summing the above equations for $n = 0,1,\ldots,m$ and applying the discrete Gronwall's inequality \cite{tao2006local}, we get
\begin{align*}
    \mathcal{E}(e^{m+1}_{\psi,N},e^{m+1}_{\phi,N},\dot{e}^{m+1}_{\phi,N}) \lesssim \frac{\tau^4}{\varepsilon^6} + \frac{h^{2m_0 - 2}}{\varepsilon^2} ,\qquad \mathcal{E}(e^{m+1}_{\psi,N},e^{m+1}_{\phi,N},\dot{e}^{m+1}_{\phi,N})  \lesssim  \varepsilon^2 + \frac{ h^{2m_0 -2}}{\varepsilon^2},
\end{align*}
which implies by the energy functional \eqref{enefun}
\begin{subequations}
\begin{align}
& \left\| e^{m+1}_{\psi,N} \right\|_{H^1} + \left\| e^{m+1}_{\phi,N} \right\|_{H^1} + \varepsilon^2 \left\| \dot{e}^{m+1}_{\phi,N} \right\|_{H^1} \lesssim h^{m_0 -1} + \frac{\tau^2}{\varepsilon^2},\\
 & \left\| e^{m+1}_{\psi,N} \right\|_{H^1} + \left\| e^{m+1}_{\phi,N} \right\|_{H^1} + \varepsilon^2 \left\| \dot{e}^{m+1}_{\phi,N} \right\|_{H^1} \lesssim h^{m_0 -1} +  \varepsilon^2.   
\end{align}\label{thm.2}
\end{subequations}
Combining \eqref{thm.2} with \eqref{Tri_err}, we prove \eqref{thm1err1}-\eqref{thm1err2} for $ n = m+1$. Thus, by taking the minimum of \eqref{thm1err1} and \eqref{thm1err2}, we obtain a uniform error bound with respect to $\varepsilon \in (0,1]$,
\begin{align*}
    \left\| e^{m+1}_{\psi} \right\|_{H^1} + \left\| e^{m+1}_{\phi} \right\|_{H^1} + \varepsilon^2 \left\| \dot{e}^{m+1}_{\phi} \right\|_{H^1} \lesssim h^{m_0 -1} + \tau.
\end{align*}

Finally, by the triangle inequality, there exist constants $\tau_2 >0$ and $h_2>0$ independent of $\varepsilon$ such that for $0<\tau \le \tau_2$ and $0<h\le h_2$
\begin{align*}
& \left\| I_N \Psi^{m+1} \right\|_{H^1} \le \left\| e_\psi^{m+1} \right\|_{H^1} + \left\| \psi(\cdot,t_{m+1}) \right\|_{H^1} \le 1 + C_\psi,\\
& \left\| I_N \Phi^{m+1} \right\|_{H^1} \le \left\| e_\phi^{m+1} \right\|_{H^1} + \left\| \phi(\cdot,t_{m+1} )\right\|_{H^1} \le 1 + C_\phi,\\
& \left\| I_N \dot{\Phi}^{m+1} \right\|_{H^1} \le \left\| \dot{e}_\phi^{m+1} \right\|_{H^1} + \left\| \partial_t \phi(\cdot,t_{m+1}) \right\|_{H^1} \le \frac{1 + C_\phi}{\varepsilon^2}.
\end{align*}
Therefore, \eqref{thm1err3} is also valid for $n = m+1$. The proof is completed by choosing $\tau_0 = \mathrm{min}\{ \tau_1,\tau_2\}$ and $h_0 = \mathrm{min}\{ h_1,h_2\}$.
\end{proof}

\begin{Remark}
Compared to the MTI-FP method with well-prepared initial data in the reference \cite{bao2017uniformly}, we choose homogeneous transmission conditions $\gamma(x) \equiv 0$ and apply the Gautschi's quadrature instead of the trapezoidal rule for discretization of integrals involving $r^n$. This brings two advantages for the MTI-FP method in this work: (i) it relaxes the regularity requirement $\phi \in C^1([0,T];H^{m_0 + 2})$ in the reference \cite{bao2017uniformly} to $\phi \in C^1([0,T];H^{m_0})$, and (ii) it greatly simplifies the scheme and the corresponding error analysis. For practical computation, it improves computational efficiency and reduces CPU time.

\end{Remark}

\begin{Remark}
 When $d = 1$, Theorem~\ref{theorem1} holds without any CFL-conditions. However, for $d = 2$ or $d = 3$, due to the use of inverse inequalities to control the $l^{\infty}$ norm of the numerical solutions \cite{bao2013optimal,shen2011spectral} , we have to impose the technical condition
  \begin{align*}
      \tau \lesssim \rho_d(h), \quad \mathrm{with} \quad\rho_d(h) = 
      \begin{cases}
          1 / \left| \mathrm{ln} h\right|, \quad & d = 2,\\
          \sqrt{h}, \quad & d = 3.
      \end{cases}
  \end{align*}
Furthermore, if under a stronger assumption (B) of the regularity, i.e., for $m_0 \ge 6$
\begin{align*}
 \mathrm{(B)} \qquad  \left\| \psi \right\|_{L^{\infty} \left( [0,T]; H^{m_0} \right)} +  \left\| \phi \right\|_{L^{\infty} \left( [0,T]; H^{m_0} \right)} + \varepsilon^2 \left\| \partial_t \phi \right\|_{L^{\infty} \left( [0,T]; H^{m_0} \right)} \lesssim 1, \nonumber 
\end{align*}
we can derive error estimates in $H^2$-norm and all of the above analysis could be extended easily. In this case, by using the following Sobolev inequalities \cite{evans2022partial}, CFL conditions are not needed for $d =1, 2, 3$,
\begin{align*}
    \left\| u \right\|_{L^{\infty}(\Omega)} \le C \left\| u \right\|_{H^2 (\Omega)},\quad 
    \left\| u \right\|_{W^{1,p}(\Omega)} \le C \left\| u \right\|_{H^2 (\Omega)}, \quad 1 < p < 6,
\end{align*}
where $\Omega$ is a bounded domain in $d$-dimensional space $(d = 1 , 2, 3)$.

\end{Remark}




\section{A multiscale interpolation in time and its uniformly accurate error bounds}\label{sec:interpolation}
In this section, we present a multiscale interpolation in time based on (i) the multiscale decomposition 
\eqref{ansatz1d} and (ii) a linear interpolation of the numerical results obtained via the MTI-FP \eqref{MTIFP}-\eqref{init00}.

\subsection{The multiscale interpolation in time}

Let $\mathcal{I}_\tau: C([0,\tau]) ({\rm or}\ {\mathbb C}^2) \rightarrow W:={\rm span}\{1, s\}$ be the linear 
interpolation operator, i.e.,
\begin{align}
&(\mathcal{I}_\tau w)(s) =\frac{\tau - s}{\tau} w(0) +\frac{s}{\tau}w(\tau), \qquad  0\le s\le \tau,  \qquad 
w\in  C([0,\tau]), \\
&(\mathcal{I}_\tau \mathbf{w})(s) =\frac{\tau - s}{\tau} w_0 +\frac{s}{\tau}w_1, \qquad  0\le s\le \tau, \qquad 
\mathbf{w}=(w_0,w_1)^T\in {\mathbb C}^2.
\end{align}
Then, we have the following estimates of this linear interpolation operator.
\begin{Proposition} For $\sigma >0$ and $\omega \in C^0([0,\tau];H^\sigma(\Omega)) $, we have
\begin{align*}
& \left\| w (\cdot,s) - (\mathcal{I}_\tau  w )(\cdot,s)\right\|_{H^\sigma} \le C_1 \tau \left\| \partial_s w \right\|_{L^\infty([0,\tau];H^\sigma)}, \\
& \left\| w (\cdot,s) -  ( \mathcal{I}_\tau w) (\cdot,s)\right\|_{H^\sigma} \le C_2 \tau^2 \left\| \partial_{ss} w \right\|_{L^\infty([0,\tau];H^\sigma)}, \qquad 0 \le s \le \tau,
\end{align*}
where $C_1$ and $C_2$ are constants independent of $\sigma$, $\tau$ and $\omega$.
\label{pro_int}
\end{Proposition}
Based on the numerical solutions from the MTI-FP \eqref{MTIFP}-\eqref{init00}, denote
\[\mathbf{v}_j^n=(v^{n,0}_j,v^{n,1}_j)^T\in{\mathbb C}^2, \quad  \mathbf{r}_j^n=(0,r^{n,1}_j)^T\in{\mathbb C}^2,
\quad j=0,1,\ldots,N, \quad n\ge0,\]
then we can present a multiscale interpolation in time for $t\ge0$ as
\begin{equation}
\mathbf{\Psi}(t) = (\Psi_0(t),\ldots,\Psi_N(t))^T, \quad \mathbf{\Phi}(t) = (\Phi_0(t),\ldots,\Phi_N(t))^T, \quad t \ge 0, \label{MTIint1}
\end{equation}
where for $j=0,1,\ldots,N$
\begin{subequations}\label{MTIint2}
\begin{align}
& \Psi_j(t_n + s) = \frac{\tau - s }{\tau} \psi^n_j+ \frac{s}{\tau} \psi_j^{n+1}, \qquad 0 \le s \le \tau, \qquad n \ge 0,\label{intpsi}\\
&\Phi_j(t_n+s) =e^{is/\varepsilon^2} (\mathcal{I}_\tau\mathbf{v}_j^n)(s)+ {\rm c.c.} + (\mathcal{I}_\tau\mathbf{r}_j^n)(s), \label{intphi}
\end{align}
\end{subequations}
with
\[ (\mathcal{I}_\tau\mathbf{v}_j^n)(s)=\frac{\tau-s}{\tau}v_j^{n,0}+\frac{s}{\tau}v_j^{n,1},\qquad 
(\mathcal{I}_\tau\mathbf{r}_j^n)(s)=\frac{s}{\tau} r_j^{n,1}, \qquad 0\le s\le \tau.
\]

\subsection{A uniformly accurate error bound}

For the multiscale interpolation \eqref{MTIint2}, we have the following error bounds.

\begin{Theorem}\label{theorem2} Under the assumption (A) and $\tau_0$, $h_0$ independent of $\varepsilon$ obtained in Theorem~\ref{theorem1}, for any $0 < \varepsilon \le 1$, when $0 < h \le h_0$ and $0 < \tau \le \tau_0$, we have for $0 \le t \le T$
\begin{equation} \label{int_err1}
\begin{aligned}
& \left\| \psi(\cdot,t) - (I_N\mathbf{\Psi})(\cdot,t) \right\|_{H^1}  \lesssim h^{m_0 -1} + \displaystyle \frac{\tau^2}{\varepsilon^2}, \\
& \left\|  \psi(\cdot,t) - (I_N\mathbf{\Psi})(\cdot,t)  \right\|_{H^1}  \lesssim h^{m_0 -1} + \tau + \varepsilon^2, 
\end{aligned}
\end{equation}
\begin{equation}
\begin{aligned}
& \left\|  \phi(\cdot,t) - (I_N\mathbf{\Phi})(\cdot,t) \right\|_{H^1}  \lesssim h^{m_0 -1} + \displaystyle \frac{\tau^2}{\varepsilon^2}, \\
& \left\|  \phi(\cdot,t) - (I_N\mathbf{\Phi})(\cdot,t)  \right\|_{H^1}  \lesssim h^{m_0 -1} + \tau +  \varepsilon^2. 
\end{aligned}\label{int_err2}
\end{equation}
Thus, by first taking the minimum of the two error estimates in \eqref{int_err1}-\eqref{int_err2} and then taking the maximum for $\varepsilon \in (0,1]$, we obtain a uniform error estimate with respect to  $\varepsilon \in (0,1]$ as
\begin{align}
& \left\|  \psi(\cdot,t) - (I_N\mathbf{\Psi})(\cdot,t) 
\right\|_{H^1}  \lesssim h^{m_0 -1} + \tau +\max_{0<\varepsilon\le 1}\min\left  \{\frac{\tau^2}{\varepsilon^2}, \varepsilon^2\right\}\lesssim h^{m_0 -1} + \tau, \label{int_uerr1} \\
&  \left\|  \phi(\cdot,t) - (I_N\mathbf{\Phi})(\cdot,t) 
\right\|_{H^1}  \lesssim h^{m_0 -1}  + \tau +\max_{0<\varepsilon\le 1}\min\left  \{\frac{\tau^2}{\varepsilon^2}, \varepsilon^2\right\}\lesssim h^{m_0 -1} + \tau. \label{int_uerr2}
\end{align}
\end{Theorem}

\begin{proof}
For $n \ge 0$ and $0 \le s \le \tau$, we define
\begin{align}
    & \Psi_N^\tau(x,t_n+s) = \mathcal{I}_\tau (P_N \psi^n) (x,s), \label{psi_exactint}\\
    & \Phi^\tau_N(x,t_n + s) = e^{is/\varepsilon^2} (\mathcal{I}_\tau (P_N v^n))(x,s) + {\rm c.c} + (\mathcal{I}_\tau (P_N r^n))(x,s), \label{phi_exactint}
\end{align}
with
\begin{align*}
\psi^n := \psi(x,t_n+\theta), \qquad v^n := v^n(x,\theta), \qquad r^n := r^n(x,\theta), \qquad 0 \le \theta \le \tau.
\end{align*}

We first prove the estimates for $I_N \mathbf{\Psi}$. Using the triangle inequality and combing interpolations \eqref{intpsi} and \eqref{psi_exactint}, we get
\begin{align*}
& \left\|  \psi(\cdot,t_n+s) -(I_N\mathbf{\Psi})(\cdot,t_n + s)\right\|_{H^1} \nonumber \\
&\le \left\| \psi(\cdot,t_n + s) -(P_N \psi)(\cdot,t_n + s)\right\|_{H^1}+ \left\| (P_N \psi)(\cdot,t_n + s) - \Psi^\tau_N (\cdot,t_n + s) \right\|_{H^1} \nonumber \\
&\quad + \left\| \Psi^\tau_N (\cdot,t_n + s) - (I_N\mathbf{\Psi})(\cdot,t_n + s) 
\right\|_{H^1} \nonumber \\
&\le \left\| \psi(\cdot,t_n + s) -(P_N \psi)(\cdot,t_n + s)\right\|_{H^1} + \left\| (I - \mathcal{I}_\tau) (P_N \psi)(\cdot,t_n + s) \right\|_{H^1} \nonumber \\
& \quad + \left\| ( P_N \psi)(\cdot,t_n) - (I_N \Psi^n)(\cdot)  \right\|_{H^1} + \left\| (P_N \psi) (\cdot,t_{n+1}) - (I_N \Psi^{n+1})(\cdot)  \right\|_{H^1}.
\end{align*}
Under the assumption (A) of $\psi$, i.e.
\begin{align*}
    \left\| \psi \right\|_{L^\infty([0,T];H^{m_0})} +  \left\| \partial_t \psi \right\|_{L^\infty([0,T];H^{m_0 -2})} + \varepsilon^2  \left\| \partial_{tt} \psi \right\|_{L^\infty([0,T];H^{m_0-4})} \lesssim 1,
\end{align*}
combining Proposition~\ref{pro_int} and error bounds \eqref{thm1err1}-\eqref{thm1err2} in Theorem~\ref{theorem1}, we obtain
\begin{align*}
&\left\|  \psi(\cdot,t_n+s) -(I_N\mathbf{\Psi})(\cdot,t_n + s)\right\|_{H^1} \\
& \lesssim h^{m_0 - 1} + \tau^2 \left\| \partial_{tt} \psi \right\|_{L^\infty([0,T];H^{1})} + \left\| e_\psi^n \right\|_{H^1} + \left\| e^{n+1}_\psi \right\|_{H^1} \lesssim h^{m_0 -1} + \frac{\tau^2}{\varepsilon^2}, 
\end{align*}
and
\begin{align*}
& \left\|  \psi(\cdot,t_n+s) -(I_N\mathbf{\Psi})(\cdot,t_n + s)\right\|_{H^1} \\
& \lesssim h^{m_0 - 1} + \tau \left\| \partial_{t} \psi \right\|_{L^\infty([0,T];H^{1})} + \left\| e_\psi^n \right\|_{H^1} + \left\| e^{n+1}_\psi \right\|_{H^1} \lesssim h^{m_0 -1} + \tau +\varepsilon^2,
\end{align*}
which proves \eqref{int_err1} and directly implies \eqref{int_uerr1}. 

Then we estimate the error bounds of $I_N \mathbf{\Phi}$. Similarly, by the triangle inequality and the assumption (A), noting \eqref{ansatz1d}, \eqref{phi_exactint} and \eqref{intphi}, we have for $0 \le s \le \tau$
\begin{align}
& \left\|  \phi(\cdot,t_n+s) -(I_N\mathbf{\Phi})(\cdot,t_n + s)\right\|_{H^1} \nonumber \\
&\le \left\| \phi(\cdot,t_n + s) -(P_N \phi)(\cdot,t_n + s)\right\|_{H^1}+ \left\| (P_N \phi)(\cdot,t_n + s) - \Phi^\tau_N (\cdot,t_n + s) \right\|_{H^1} \nonumber \\
&\quad + \left\| \Phi^\tau_N (\cdot,t_n + s) - (I_N\mathbf{\Phi})(\cdot,t_n + s) 
\right\|_{H^1} \nonumber \\
&\lesssim h^{m_0 -1} + \left\| (I - \mathcal{I}_\tau) (P_N v^n)(\cdot, s) \right\|_{H^1} + \left\| (I - \mathcal{I}_\tau) (P_N r^n)(\cdot, s) \right\|_{H^1} \nonumber \\
& \quad + \left\| ( P_N v^n)(\cdot,0) - (I_N \mathbf{v}^{n,0})(\cdot)  \right\|_{H^1} + \left\| ( P_N v^n)(\cdot,\tau) - (I_N \mathbf{v}^{n,1})(\cdot)  \right\|_{H^1} \nonumber \\
& \quad + \left\| (P_N r^n) (\cdot,\tau) - (I_N \mathbf{r}^{n,1})(\cdot)  \right\|_{H^1}. \label{thm2P1}
\end{align}
Noticing Proposition~\ref{pro_int} and the prior estimate \eqref{lm1v} of $v^n$, we obtain for $0\le s\le \tau$
\begin{equation}\label{thm2P2}
\begin{aligned}
&\left\| v^n(\cdot, s) - (\mathcal{I}_\tau v^n)(\cdot,s)\right\|_{H^1} \lesssim \tau \left\| \partial_s v^n \right\|_{L^{\infty}([0,\tau];H^1)} \lesssim \tau,  \\
& \left\| v^n(\cdot, s) - (\mathcal{I}_\tau v^n)(\cdot,s) \right\|_{H^1} \lesssim \tau^2 \left\| \partial_{ss} v^n \right\|_{L^{\infty}([0,\tau];H^1)} \lesssim \displaystyle \frac{\tau^2}{\varepsilon^2}.
\end{aligned}
\end{equation}
Similarly, for $ 0 \le s \le \tau$, we have 
\begin{align}
\left\| r^n(\cdot, s) - (\mathcal{I}_\tau r^n)(\cdot,s) \right\|_{H^1} \lesssim \tau, \quad  \left\| r^n(\cdot, s) -(\mathcal{I}_\tau r^n)(\cdot,s) \right\|_{H^1} \lesssim \displaystyle \frac{\tau^2}{\varepsilon^2}. \label{thm2P3}
\end{align}
From the initial data in \eqref{v1d}, \eqref{MTIIni1} and the error estimates \eqref{thm1err1}-\eqref{thm1err2}, we immediately obtain
\begin{equation}\label{thm2P4}
\begin{aligned}
&\left\|(P_Nv^n)(\cdot,0)-(I_N\mathbf{v}^{n,0})(\cdot)\right\|_{H^1}\lesssim h^{m_0 -1}+ \frac{\tau^2}{\varepsilon^2},\\
&\left\|(P_Nv^n)(\cdot,0)-(I_N\mathbf{v}^{n,0})(\cdot)\right\|_{H^1}\lesssim h^{m_0 -1}+ \varepsilon^2.
\end{aligned}
\end{equation}
Subtracting the Fourier coefficients $\widehat{(v^n)}_l(\tau)$ in \eqref{Sv} from $\widetilde{(\mathbf{v}^{n,1})}_l$ in \eqref{MTIFP.Cv}, noting \eqref{thm2P4} and using  $|a_l(\tau)| \lesssim1$, we get
\begin{equation}
\begin{aligned}
\left\| (P_N v^n)(\cdot,\tau) - (I_N \mathbf{v}^{n,1}) (\cdot) \right\|_{H^1}&\lesssim  \left\| (P_N v^n )(\cdot,0) - ( I_N\mathbf{v}^{n,0}) (\cdot) \right\|_{H^1} \\
& \lesssim  h^{m_0 -1} + \frac{\tau^2}{\varepsilon^2}, \\
\left\| (P_N v^n)(\cdot,\tau) - (I_N \mathbf{v}^{n,1}) (\cdot) \right\|_{H^1} & \lesssim h^{m_0 -1} + \varepsilon^2. 
\end{aligned}
\label{thm2P5}
\end{equation}
Similarly, noting the formulation of $\xi^n_\phi$ in \eqref{Localphi} and $\eta^n_\phi$ in \eqref{FNLphi}, we subtract the Fourier coefficients $\widehat{(r^n)}_l(\tau)$ in \eqref{Sr} from $\widetilde{(\mathbf{r}^{n,1})}_l$ in \eqref{MTIFP.Cr} and obtain
\begin{align}
 \left\| (P_N r^n)(\cdot,\tau) - (I_N \mathbf{r}^{n,1}) (\cdot) \right\|_{H^1} = \left\| \xi^n_\phi  + \eta_\phi^n \right\|_{H^1} \le \left\| \xi^n_\phi \right\|_{H^1} + \left\| \eta^n_\phi \right\|_{H^1} \label{thm2P6}.
\end{align}
Plugging the energy estimates of error functions \eqref{est_lol} and \eqref{lmNon_est} into \eqref{thm2P6}, it arrives at
\begin{equation}\label{thm2P7}
\begin{aligned}
& \left\| (P_N r^n)(\cdot,\tau) - (I_N \mathbf{r}^{n,1}) (\cdot) \right\|_{H^1}\lesssim   h^{m_0 -1}+ \frac{\tau^2}{\varepsilon^2},\\
& \left\| (P_N r^n)(\cdot,\tau) - (I_N \mathbf{r}^{n,1}) (\cdot) \right\|_{H^1}\lesssim
h^{m_0 -1} + \varepsilon^2 .
\end{aligned}
\end{equation}
Plugging \eqref{thm2P2}-\eqref{thm2P5} and \eqref{thm2P7} into \eqref{thm2P1}, we obtain the error estimates in 
\eqref{int_err2}, which directly implies \eqref{int_uerr2}. 
\end{proof}


\section{Numerical results}\label{sec:numerical}
In this section, we report numerical results to verify our error bounds and apply the MTI-FP method for numerically studying different limiting models of the KGS system \eqref{KGSndim}.

\subsection{Accuracy test}
We consider a 1D example and take the following initial data for \eqref{KGSndim}
\begin{align}
\psi_0(x) = \frac{1+i}{2}\mathrm{sech}(x^2), \quad \phi_0(x)  = \frac{1}{2} e^{-x^2},\quad \phi_1(x) = \frac{1}{\sqrt{2}}e^{-x^2}, \quad x \in {\mathbb R}, \label{nume_ini}
\end{align}
Practically, the problem is truncated on an interval $\Omega = (-16,16)$, which is large enough such that the truncation error of \eqref{psi1d}-\eqref{phi1d} to the original whole space problem \eqref{KGSndim} is negligible. To quantify numerical errors, we introduce the error functions 
\begin{align*}
&e^{\varepsilon}_{\psi}(t_n) = \left\| \psi(\cdot,t_n) - I_N \Psi^{n} \right\|_{H^1}, \quad e^{*}_{\psi}(t_n) = \max_{0 < \varepsilon \le 1} \{  e^{\varepsilon}_{\psi}(t_n)\},\\
&e^{\varepsilon}_{\phi} (t_n) = \left\| \phi(\cdot,t_n) - I_N \Phi^{n} \right\|_{H^1},\quad e^{*}_{\phi}(t_n) = \max_{0 < \varepsilon \le 1} \{  e^{\varepsilon}_{\phi}(t_n)\},\quad n \ge 0.
\end{align*}

Since the exact solution is not available, the ``exact'' solution is computed numerically by the MTI-FP method under well-prepared initial data \cite{bao2017uniformly} with very small mesh size $h = h_e =  1/32$ and time step $\tau = \tau_e = 1 \times 10^{-6}$. Tables~\ref{t1} \& \ref{t2} show the spatial errors for $\psi$ and $\phi$ at $t_n=1$ under different $\varepsilon$ and $h$ with a very small time step $\tau = \tau_e$ such that the temporal discretization error is negligible. Tables~\ref{t3} \& \ref{t4} show the temporal errors for $\psi$ and $\phi$ at $t_n=1$ under different $\varepsilon$ and $\tau$ with a very fine mesh size $h = h_e$ such that the spatial discretization error is negligible. In addition, Fig.~\ref{fig:interpolation} plots the multiscale interpolation errors $e_\psi^*(T)$ and $e_\phi^*(T)$ at $T = 1.679$ obtained via the interpolation \eqref{MTIint1}-\eqref{MTIint2} for different $\tau$ with $h = h_e$.

\begin{table}[ht!]
\caption{Spatial errors for $e^{\varepsilon}_{\psi}(t=1)$ with $\tau = 1 \times 10^{-6}$ for different $\varepsilon$ and $h$.} \label{t1}
\centering
\begin{tabular}{@{}llllll@{}} 
\hline
{$e^{\varepsilon}_{\psi}(1)$}\quad\quad \quad& {$h_0 = 1$} \quad\quad \quad & {$ h_0 / 2$} \quad\quad\quad& {$h_0 / 2^2$} \quad\quad\quad  & {$h_0 / 2^3$} \quad\quad \quad  & {$h_0 / 2^4$} \quad \quad  \\ \hline
{$\varepsilon_0=0.5$} & 2.83E-01 & 2.57E-02 & 1.44E-04 & 9.80E-09 & 2.26E-10 \\ {$\varepsilon_0/2^1$} & 2.45E-01 & 2.63E-02 & 1.42E-04 & 9.77E-09 & 3.64E-12 \\ {$\varepsilon_0/2^2$} & 2.48E-01 & 2.40E-02 & 1.56E-04 & 9.77E-09 & 6.25E-12 \\
{$\varepsilon_0 / 2^3$} & 2.49E-01 & 2.41E-02 & 1.51E-04 & 9.71E-09 & 2.54E-11 \\
{$\varepsilon_0 / 2^4$} & 2.48E-01 & 2.41E-02 & 1.51E-04 & 9.72E-09 & 1.02E-10 \\
{$\varepsilon_0 / 2^5$} & 2.48E-01 & 2.41E-02 & 1.51E-04 & 9.72E-09 & 4.10E-10 \\
{$\varepsilon_0 / 2^6$} & 2.48E-01 & 2.41E-02 & 1.51E-04 & 9.85E-09 & 1.64E-09 \\
{$\varepsilon_0 / 2^7$} & 2.48E-01 & 2.41E-02 & 1.51E-04 & 1.17E-08 & 6.55E-09 \\
{$\varepsilon_0 / 2^8$} & 2.48E-01 & 2.41E-02 & 1.51E-04 & 2.78E-08 & 2.61E-08 \\
{$\varepsilon_0 / 2^9$} & 2.48E-01 & 2.41E-02 & 1.51E-04 & 9.61E-08 & 9.56E-08 \\
\hline
\end{tabular}
\end{table}

\begin{table}[ht!]
\caption{Spatial errors for $e^{\varepsilon}_{\phi}(t=1)$ with $\tau = 1 \times 10^{-6}$ for different $\varepsilon$ and $h$.}
\label{t2}
\centering
\begin{tabular}{@{}llllll@{}}
\hline
{$e^{\varepsilon}_{\phi}(1)$}\quad\quad  \quad&
{$h_0 = 1$} \quad\quad\quad  &
{$ h_0 / 2$} \quad\quad\quad &
{$h_0 / 2^2$} \quad\quad\quad  &
{$h_0 / 2^3$} \quad\quad \quad & {$h_0 / 2^4$} \quad\quad \\
\hline
{$\varepsilon_0=0.5$} & 3.50E-02 & 3.73E-04 & 1.03E-06 & 1.14E-10 & 1.03E-10 \\
{$\varepsilon_0/2^1$} & 8.53E-02 & 1.92E-03 & 3.82E-06 & 1.31E-10 & 1.18E-10 \\
{$\varepsilon_0/2^2$} & 1.25E-01 & 2.56E-04 & 7.69E-06 & 2.03E-10 & 2.02E-10 \\
{$\varepsilon_0 / 2^3$} & 9.40E-02 & 1.98E-04 & 2.74E-06 & 2.67E-10 & 2.13E-10 \\
{$\varepsilon_0 / 2^4$} & 9.83E-02 & 3.61E-05 & 5.72E-07 & 1.49E-10 & 1.24E-10 \\
{$\varepsilon_0 / 2^5$} & 4.23E-02 & 8.50E-05 & 1.99E-07 & 1.09E-10 & 8.84E-11 \\
{$\varepsilon_0 / 2^6$} & 1.43E-01 & 7.10E-05 & 5.18E-08 & 1.09E-10 & 7.67E-11 \\
{$\varepsilon_0 / 2^7$} & 2.66E-02 & 9.50E-05 & 5.48E-09 & 9.47E-11 & 6.81E-11 \\
{$\varepsilon_0 / 2^8$} & 1.21E-01 & 8.93E-06 & 2.31E-09 & 1.10E-10 & 7.17E-11 \\
{$\varepsilon_0 / 2^9$} & 1.36E-01 & 3.99E-05 & 7.76E-10 & 1.10E-10 & 7.31E-11 \\
\hline
\end{tabular}
\end{table}

\begin{table}[ht!]
\caption{Temporal errors for $e^{\varepsilon}_{\psi}(t=1)$ with $h = 1/32$ for different $\varepsilon$ and $\tau$.}
\label{t3}
\centering
\begin{tabular}{@{}lllllll@{}}
\hline
{$e^{\varepsilon}_{\psi}(1)$}\quad\quad &
{$\tau_0=0.1$} \quad\quad&
{$\tau_0/2^2$} \quad\quad&
{$\tau_0/2^4$} \quad\quad&
{$\tau_0/2^6$} \quad\quad&
{$\tau_0/2^8$} \quad\quad&
{$\tau_0/2^{10}$} \\
\hline
{$\varepsilon_0=1$} & 1.26E-02 & 8.20E-04 & 4.87E-05 & 3.00E-06 & 1.87E-07 & 1.17E-08 \\
rate & {---} & 1.97 & 2.04 & 2.01 & 2.00 & 2.00 \\
\hline
{$\varepsilon_0/2^1$} & 1.77E-02 & 1.16E-03 & 6.93E-05 & 4.28E-06 & 2.66E-07 & 1.66E-08 \\
rate & {---} & 1.96 & 2.03 & 2.01 & 2.00 & 2.00 \\
\hline
{$\varepsilon_0/2^2$} & 1.29E-02 & 7.93E-04 & 5.05E-05 & 3.16E-06 & 1.98E-07 & 1.24E-08 \\
rate & {---} & 2.01 & 1.99 & 2.00 & 2.00 & 2.00 \\
\hline
{$\varepsilon_0 / 2^3$} & 1.42E-02 & 1.51E-03 & 8.93E-05 & 5.49E-06 & 3.42E-07 & 2.13E-08 \\
rate & {---} & 1.62 & 2.04 & 2.01 & 2.00 & 2.00 \\
\hline
{$\varepsilon_0 / 2^4$} & 2.64E-03 & 2.76E-03 & 3.57E-04 & 2.15E-05 & 1.32E-06 & 8.25E-08 \\
rate & {---} & {---} & 1.48 & 2.03 & 2.01 & 2.00 \\
\hline
{$\varepsilon_0 / 2^5$} & 3.59E-04 & 4.17E-04 & 4.11E-04 & 8.77E-05 & 5.30E-06 & 3.27E-07 \\
rate & {---} & {---} & {---} & 1.11 & 2.02 & 2.01 \\
\hline
{$\varepsilon_0 / 2^6$} & 1.44E-04 & 6.61E-05 & 9.67E-05 & 9.74E-05 & 2.18E-05 & 1.32E-06 \\
rate & {---} &{---} & {---} & {---} & 1.08 & 2.02 \\
\hline
{$\varepsilon_0 / 2^{7}$} & 2.24E-05 & 2.25E-05 & 1.62E-05 & 2.39E-05 & 2.41E-05 &  5.45E-06 \\
rate & {---} & {---} & {---} & {---} & {---} & 1.07 \\
\hline
{$\varepsilon_0 / 2^{8}$} & 1.28E-05 & 4.69E-06 & 5.49E-06 & 4.02E-06 & 5.96E-06 & 6.01E-07 \\
rate & {---} & {---} & {---} & {---} & {---} & {---} \\
\hline \hline
{$e^{*}_{\psi}(1)$} & 1.77E-02 & 2.76E-03 & 4.11E-04 & 9.74E-05 & 2.41E-05 & 6.01E-06 \\
rate & {---} & 1.34 & 1.37 & 1.04 & 1.01 & 1.00 \\
\hline
\end{tabular}
\end{table}

\begin{table}[ht!]
\caption{Temporal errors for $e^{\varepsilon}_{\phi}(t=1)$ with $h = 1/32$ for different $\varepsilon$ and $\tau$.}
\label{t4}
\centering
\begin{tabular}{@{}lllllll@{}}
\hline
{$e^{\varepsilon}_{\phi}(1)$}\quad\quad &
{$\tau_0=0.1$} \quad\quad&
{$\tau_0/2^2$} \quad\quad&
{$\tau_0/2^4$} \quad\quad&
{$\tau_0/2^6$} \quad\quad&
{$\tau_0/2^8$} \quad\quad&
{$\tau_0/2^{10}$} \\
\hline
{$\varepsilon_0=1$} & 2.18E-03 & 7.42E-05 & 3.71E-06 & 2.25E-7 & 1.41E-08 & 9.48E-10 \\
rate & {---} & 2.44 & 2.16 & 2.02 & 2.00 & 1.95 \\
\hline
{$\varepsilon_0/2^1$} & 5.94E-03 & 2.81E-04 & 1.60E-05 & 9.85E-07 & 6.14E-08 & 3.84E-09 \\
rate & {---} & 2.20 & 2.07 & 2.01 & 2.00 & 2.00 \\
\hline
{$\varepsilon_0/2^2$} & 9.56E-03 & 6.80E-04 & 3.94E-05 & 2.41E-06 & 1.50E-07 & 9.35E-09 \\
rate & {---} & 1.91 & 2.06 & 2.02 & 2.00 & 2.00 \\
\hline
{$\varepsilon_0 / 2^3$} & 3.46E-03 & 1.81E-04 & 1.28E-05 & 8.05E-07 & 5.03E-08& 3.14E-09 \\
rate & {---} & 2.13 & 1.91 & 2.00 & 2.00 & 2.00 \\
\hline
{$\varepsilon_0 / 2^4$} & 1.29E-03 & 3.58E-04 & 6.43E-06 & 4.18E-07 & 2.61E-08 & 1.64E-09 \\
rate & {---} & 0.93 & 2.90 & 1.97 & 2.00 & 2.00 \\
\hline
{$\varepsilon_0 / 2^5$} & 2.13E-04 & 8.38E-05 & 2.23E-05 & 3.40E-07 & 2.28E-08 & 1.45E-09 \\
rate & {---} & 0.67 & 0.95 & 3.02 & 1.95 & 1.99 \\
\hline
{$\varepsilon_0 / 2^6$} & 8.59E-05 & 3.54E-06 & 9.42E-07 & 2.69E-07 & 2.07E-08 & 1.34E-09 \\
rate & {---} & 2.30 & 0.96 & 0.90 & 1.85 & 1.98 \\
\hline
{$\varepsilon_0 / 2^{7}$} & 1.28E-05 & 1.77E-06 & 6.00E-08 & 1.53E-08 & 7.53E-09 & 1.30E-09 \\
rate & {---} & 1.42 & 2.44 & 0.99 & 0.51 & 1.27 \\
\hline
{$\varepsilon_0 / 2^{8}$} & 4.61E-06 & 5.56E-07 & 2.09E-08 & 1.04E-09 & 4.75E-10 & 4.01E-10 \\
rate & {---} & 1.53 & 2.37 & 2.16 & 0.57 & 0.12 \\
\hline \hline
{$e^{*}_{\phi}(1)$} & 9.56E-03 & 6.80E-04 & 3.94E-05 & 2.41E-06 & 1.50E-07 & 9.35E-09 \\
rate & {---} & 1.91 & 2.06 & 2.02 & 2.00 & 2.00 \\
\hline
\end{tabular}
\end{table}

From Tables~\ref{t1}--\ref{t4} and Fig.~\ref{fig:interpolation}, we can draw the following conclusions:
\begin{itemize}
\item The MTI-FP method \eqref{MTIFP}-\eqref{init00} achieves spectral accuracy in space when the solution is smooth, which is uniformly for $\varepsilon \in (0,1]$ (cf. Tables~\ref{t1} $\&$ \ref{t2}). It is second-order accurate in time when $0 < \tau \lesssim \varepsilon^2$ (cf. Tables~\ref{t3} $\&$ \ref{t4} upper triangle) and the error is at $O(\varepsilon^2)$ when $0 < \varepsilon \lesssim \sqrt{\tau}$ , which is independent of $\tau$ (cf. Tables~\ref{t3} $\&$ \ref{t4} lower triangle). The method obtains uniformly first-order accuracy in time for $\varepsilon \in (0,1]$ (cf. Tables~\ref{t3} $\&$ \ref{t4} last row). In addition, the practical temporal uniform convergence rate of the MTI-FP method for $\phi$ is second order, better than the theoretical result which is limited by the nonlinear coupling with $\psi$ and analysis techniques. All these numerical results confirm our error bounds in Theorem~\ref{theorem1}.
\item The multiscale interpolation \eqref{MTIint1}-\eqref{MTIint2} achieves first-order temporal accuracy for any $t >0$, which is uniform with respect to $\varepsilon \in (0,1]$. The results show that our error bounds in Theorem~\ref{theorem2} is sharp.
\end{itemize} 

\begin{figure}[ht!]
\centerline{\includegraphics[width=11cm,height=5cm]{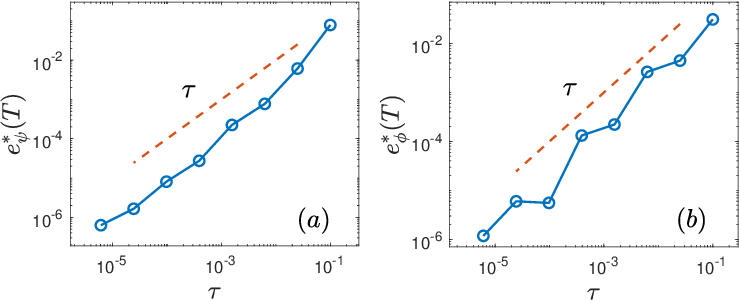}}
\caption{Multiscale interpolation error of $e_\psi^*(T)$ (left) and $e_\phi^*(T)$ (right) at $T = 1.679$ under smooth initial data \eqref{nume_ini}.}
\label{fig:interpolation}
\end{figure}


\subsection{Convergence rates from KGS to its limiting models}
In this subsection, we apply the proposed MTI-FP method \eqref{MTIFP}-\eqref{init00} to numerically study the
convergence rates from the KGS system \eqref{KGSndim} to its limiting model---the Schr\"odinger equations \eqref{lim}-\eqref{lim_initial}, as well as the semi-limiting model---the Schr\"odinger equations with wave operator \eqref{slim}-\eqref{slim_initial} under different initial data $\gamma(\bx)$. 

We consider the 1D example and denote by $(\psi,\phi)$ the solution of the KGS equations \eqref{psi1d}-\eqref{phi1d}, which is numerically solved by the proposed MTI-FP method on a bounded domain $\Omega=(-16,16)$ with a very fine mesh $h = 1/32$ and a small time step $\tau = 1 \times 10^{-5}$. Let $(\psi_{_{\rm S}}, v_{_{\rm S}})$ be the solution of the Schr\"odinger equations \eqref{lim}-\eqref{lim_initial} and $(\psi_{_{\rm SW}}, v_{_{\rm SW}})$ be the solution of the Schr\"odinger equations with wave operator \eqref{slim}-\eqref{slim_initial}, which are both obtained numerically by the exponential wave integrator Fourier pseudospectral (EWI-FP) method \cite{bao2024optimal,hochbruck2010exponential} under the same set-up with $\lambda =\mu =1$. Based on the asymptotic performance \eqref{ansatzRd}, define the solution $\phi$ from the two limiting models as
\begin{align*}
&\phi_{_{\rm S}}(x,t) = e^{it/\varepsilon^2} v_{_{\rm S}}(x,t) + e^{-it/\varepsilon^2} \overline{v_{_{\rm S}}(x,t)},\qquad t \ge 0, \\
& \phi_{_{\rm SW}}(x,t) = e^{it/\varepsilon^2} v_{_{\rm SW}}(x,t) + e^{-it/\varepsilon^2} \overline{v_{_{\rm SW}}(x,t)}.
\end{align*} 
Introduce the following error functions 
\begin{align*}
    & e_{_{\rm S}}(t)  = e_{_{\rm S}}^\psi(t)  + e_{_{\rm S}}^\phi (t), \qquad   e_{_{\rm SW}}(t)  = e_{_{\rm SW}}^\psi(t)  + e_{_{\rm SW}}^\phi (t), \qquad t \ge 0,
\end{align*}
where
\begin{align*}
& e_{_{\rm S}}^\psi(t) := \left\| \psi(\cdot,t) - \psi_{_{\rm S}}(\cdot,t) \right\|_{H^1}, &&  e_{_{\rm S}}^\phi(t) := \left\| \phi(\cdot,t) - \phi_{_{\rm S}}(\cdot,t) \right\|_{H^1},\\
& e_{_{\rm SW}}^\psi(t) := \left\| \psi(\cdot,t) - \psi_{_{\rm SW}}(\cdot,t) \right\|_{H^1}, &&  e_{_{\rm SW}}^\phi(t) := \left\| \phi(\cdot,t) - \phi_{_{\rm SW}}(\cdot,t) \right\|_{H^1}.
\end{align*}

Fig.~\ref{fig:lim_smooth} plots the errors $e_{_{\rm S}}(t)$ with smooth initial data \eqref{nume_ini} and Fig.~\ref{fig:semi_smooth} depicts the errors $e_{_{\rm SW}}(t)$ under the same set-up with different choices of $\gamma(x)$ in \eqref{slim_initial}, i.e., (a) homogeneous initial data and (b) well-prepared initial data. In addition, we also study the convergence rates of the KGS system \eqref{KGSndim} with $d = \lambda = \mu = 1$ when the initial data \eqref{KGSini} is not smooth. For this purpose, we consider the following non-smooth initial data:
\begin{align}
\psi_0(x) = \frac{1 + i}{2} x^m |x| \mathrm{sech}(x^2),\quad \phi_0(x) =  \frac{1}{2} x^m |x| e^{-x^2}, \quad \phi_1(x) = \frac{1}{\sqrt{2}} x^m |x| e^{-x^2}.\label{non_smooth}
\end{align}
Figs.~\ref{fig:limi_H3} \& \ref{fig:limi_H4} depict the errors $e_{_{\rm S}}^\psi(t)$ and $e_{_{\rm S}}^\phi(t)$ when $m=3,4$ in the non-smooth initial data \eqref{non_smooth} and Figs.~\ref{fig:semi_H2} \& \ref{fig:semi_H3} show the evolution of $e_{_{\rm SW}}^\psi(t)$ and $e_{_{\rm SW}}^\phi(t)$ when $m = 2,3$ in \eqref{non_smooth}.

\begin{figure}[ht!]
\centerline{\includegraphics[width=10cm,height=4.5cm]{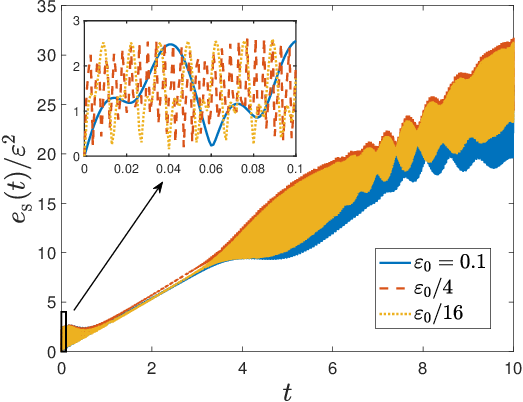}}
\caption{Convergence of the KGS \eqref{KGSndim} to its limiting model \eqref{lim} in 1D with smooth initial data \eqref{nume_ini}.}
\label{fig:lim_smooth}
\end{figure}

\begin{figure}[ht!]
\centerline{\includegraphics[width=11cm,height=4.5cm]{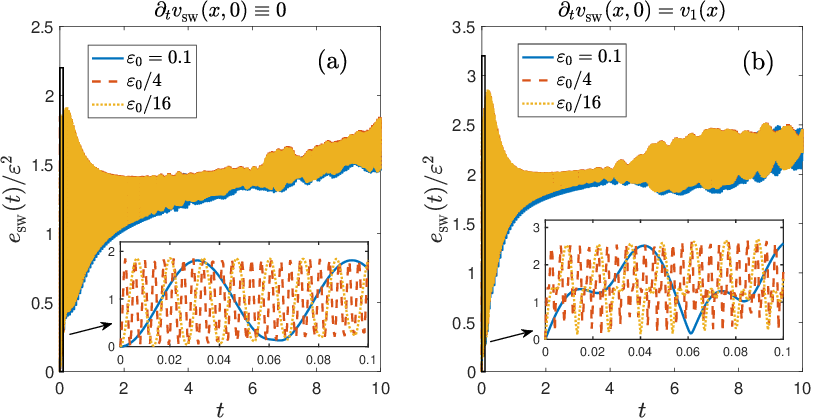}}
\caption{Convergence of the KGS \eqref{KGSndim} to its semi-limiting model \eqref{slim} in 1D with smooth initial data \eqref{nume_ini}, under different $\gamma$ in \eqref{slim_initial}: (a) $\gamma(x) \equiv 0$, and (b) $\gamma(x)$ is taken as the well-prepared initial data $v_1(x)$ in \eqref{wellprepared}.}
\label{fig:semi_smooth}
\end{figure}

\begin{figure}[ht!]
\centerline{\includegraphics[width=11cm,height=5cm]{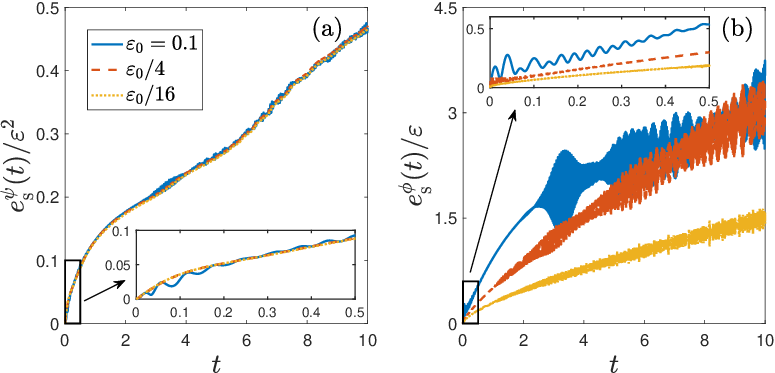}}
\caption{Convergence of the KGS \eqref{KGSndim} to its limiting model \eqref{lim} in 1D with non-smooth initial data $m = 2$ in \eqref{non_smooth}.}
\label{fig:limi_H3}
\end{figure}

\begin{figure}[ht!]
\centerline{\includegraphics[width=11cm,height=5cm]{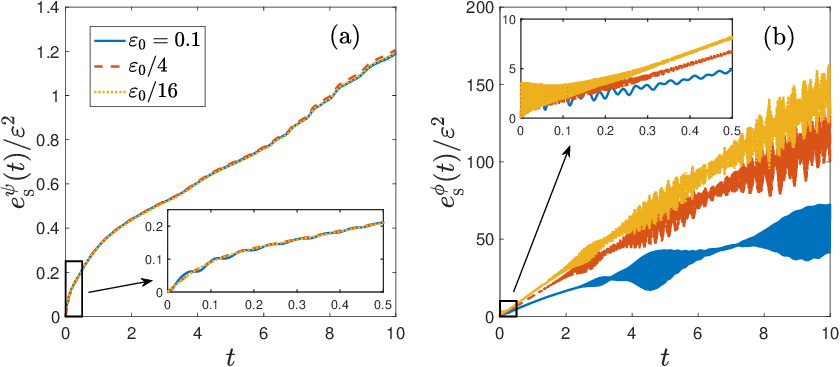}}
\caption{Convergence of the KGS \eqref{KGSndim} to its limiting model \eqref{lim} in 1D with non-smooth initial data $m = 3$ in \eqref{non_smooth}.}
\label{fig:limi_H4}
\end{figure}

\begin{figure}[ht!]
\centerline{\includegraphics[width=11cm,height=7cm]{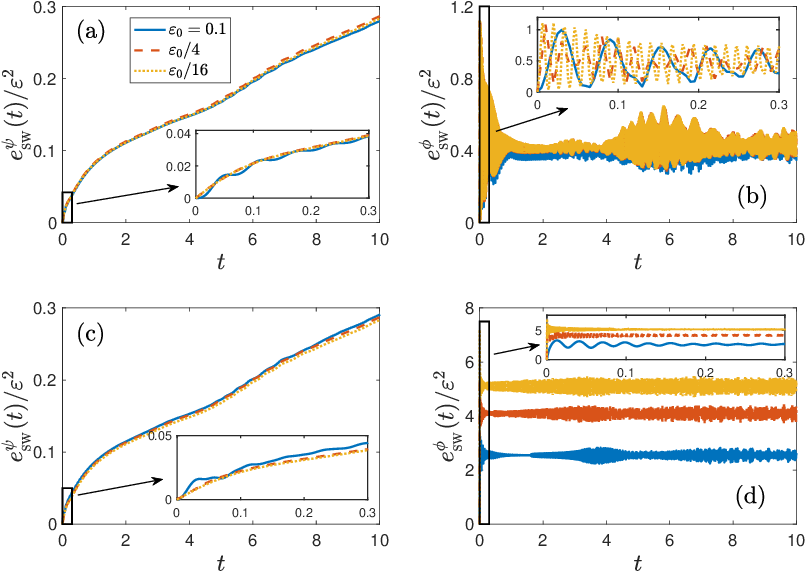}}
\caption{Convergence of the KGS \eqref{KGSndim} to its semi-limiting model \eqref{slim} in 1D with non-smooth initial data $m = 1$ in \eqref{non_smooth}, under different $\gamma$: (a)-(b) $\gamma(x) \equiv 0$ (upper), and (c)-(d) $\gamma(x)$ is taken as the well-prepared initial data $v_1(x)$ in \eqref{wellprepared} (lower).}
\label{fig:semi_H2}
\end{figure}

\begin{figure}[ht!]
\centerline{\includegraphics[width=11cm,height=7cm]{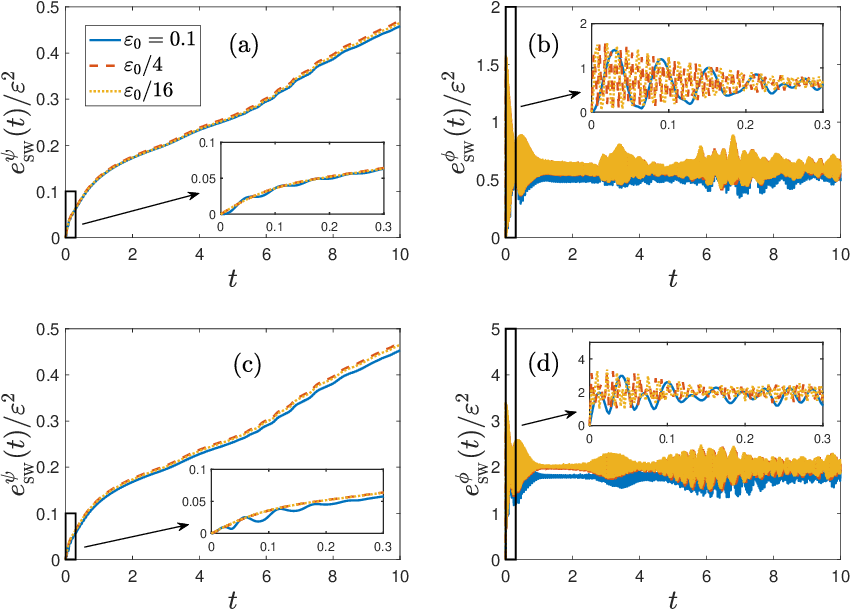}}
\caption{Convergence of the KGS \eqref{KGSndim} to its semi-limiting model \eqref{slim} in 1D with non-smooth initial data $m = 2$ in \eqref{non_smooth}, under different $\gamma$: (a)-(b) $\gamma(x) \equiv 0$ (upper), and (c)-(d) $\gamma(x)$ is taken as the well-prepared initial data $v_1(x)$ in \eqref{wellprepared} (lower).}
\label{fig:semi_H3}
\end{figure}


From Figs.~\ref{fig:lim_smooth}--\ref{fig:semi_H3}, we can draw the following conclusions:
\begin{itemize}
\item The solution of the KGS system \eqref{KGSndim} converges to that of the limiting model---the Schr\"odinger equations \eqref{lim} quadratically in $\varepsilon$ (not uniform in time), provided that the initial data \eqref{KGSini} satisfies: $\psi_0$, $\phi_0$ and $\phi_1 \in H^4(\mathbb{R}^d)$, i.e., for $0 \le t \le T$, 
\begin{align*}
    \left\| \psi(\cdot,t) - \psi_{_{\rm S}}(\cdot,t) \right\|_{H^1(\Omega)} + \left\| \phi(\cdot,t) - \phi_{_{\rm S}}(\cdot,t) \right\|_{H^1(\Omega)} \le (C_1 + C_2 T) \varepsilon^2,
\end{align*}
where $C_1$, $C_2>0$ are two constants independent of $\varepsilon$ and $T$ (cf. Figs.~\ref{fig:lim_smooth} $\&$ \ref{fig:limi_H4}). On the contrary, if under weaker regularity assumption, i.e., $\psi_0$, $\phi_0$ and $\phi_1 \in H^3$, the convergence rate of $\phi$ in $\varepsilon$ collapses to linear rate while that of $\psi$ remains the same, i.e.,
\begin{align*}
& \left\| \psi(\cdot,t) - \psi_{_{\rm S}}(\cdot,t) \right\|_{H^1(\Omega)} \le (C_1 + C_2 T) \varepsilon^2,\quad 0 \le t \le T, \\
& \left\| \phi(\cdot,t) - \phi_{_{\rm S}}(\cdot,t) \right\|_{H^1(\Omega)} \le (C_1 + C_2 T) \varepsilon, 
\end{align*}
where $C_3$, $C_4>0$ are two constants independent of $\varepsilon$ and $T$ (cf. Fig.~\ref{fig:limi_H3}).
\item The solution of the KGS system \eqref{KGSndim} converges to that of the semi-limiting model---the Schr\"odinger equations with wave operator \eqref{slim} quadratically in $\varepsilon$ (not uniform in time for $\psi$ and uniform in time for $\phi$), provided that the initial data \eqref{KGSini} satisfies: $\psi_0$, $\phi_0$ and $\phi_1 \in H^2(\mathbb{R}^d)$, i.e.,
\begin{align*}
& \left\| \psi(\cdot,t) - \psi_{_{\rm SW}}(\cdot,t) \right\|_{H^1(\Omega)} \le (C_1 + C_2 T) \varepsilon^2,\quad 0 \le t \le T, \\
& \left\| \phi(\cdot,t) - \phi_{_{\rm SW}}(\cdot,t) \right\|_{H^1(\Omega)} \le C_0 \varepsilon^2, 
\end{align*}
where $C_0$, $C_1$, $C_2>0$ are three constants independent of $\varepsilon$ and $T$ (cf. Figs.~\ref{fig:semi_smooth} $\&$ \ref{fig:semi_H2} $\&$ \ref{fig:semi_H3}).
\item For different choices of $\gamma(\bx)$ in \eqref{slim_initial}, the KGS system \eqref{KGSndim}-\eqref{KGSini} converges to its semi-limiting model---the Schr\"odinger equations with wave operator \eqref{slim} and the constant $C_0$ depends on $\gamma$. In general, when $\gamma(\bx) \equiv0$, the constant $C_0$ becomes the minimum among different choices of $\gamma$, which implies better asymptotic performance in $\varepsilon$ (cf. Figs.~\ref{fig:semi_H2} $\&$ \ref{fig:semi_H3}).
\end{itemize}


\section{Conclusion}\label{sec:conclusion}
In this paper, we proposed a simplified uniformly accurate multiscale time integrator Fourier pseudospectral (MTI-FP) method for the Klein-Gordon-Schr\"odinger (KGS) equation in the nonrelativistic limit regime with a dimensionless parameter $0 < \varepsilon \le 1$. This method was designed based on two ingredients: (i) a multiscale decomposition by frequency in each time interval with simplified transmission conditions, and (ii) an exponential wave integrator for temporal discretization and a Fourier pseudospectral method for spatial discretization. We established optimal error bounds for the MTI-FP method, which is uniformly first-order accurate in time. In addition, a multiscale interpolation in time was presented to obtain a uniformly accurate approximation of the solution at any time $t \ge 0$ by adopting a linear interpolation of the micro-variables within each time interval. Numerical results were reported to confirm our error bounds and to demonstrate the accuracy and efficiency of this method as well as to show the convergence rates of the KGS to its different limiting models.

\section*{Acknowledgment}
The work of the first author was partially supported by National Natural Science Foundation of China (No. 12401539). The work of the second author was funded by the Fog Research Institute under contract No.~FRI-454.

\bibliographystyle{plain}
\bibliography{references}

\end{document}